\documentclass[12pt]{amsart}
\usepackage{amssymb}
\usepackage{amsmath, graphicx, rotating}
\usepackage{color}
\usepackage{soul}
\usepackage{todonotes}

\usepackage{dsfont}
\usepackage{mathrsfs}

\usepackage[T1]{fontenc}
\usepackage{lmodern}

\usepackage[english]{babel}

\usepackage{ upgreek }
\usepackage{amsmath,amssymb}
\usepackage{amsthm}
\usepackage{graphicx}
\usepackage{float}

\usepackage{ bbm }
\usepackage{ stmaryrd }
\usepackage{ mathrsfs }
\usepackage{frcursive}

\usepackage{pgf, tikz}

\definecolor{rouge}{rgb}{0.7,0.00,0.00}
\definecolor{vert}{rgb}{0.00,0.5,0.00}
\definecolor{bleu}{rgb}{0.00,0.00,0.8}

\newtheorem{theorem}{Theorem}[section]
\newtheorem{lemma}[theorem]{Lemma}

\newtheorem{corollary}[theorem]{Corollary}

\numberwithin{equation}{section}

\def\bb#1{\mathbb{#1}}

\def\geq{\geqslant}
\def\leq{\leqslant}

\def\geq{\geqslant}
\def\leq{\leqslant}

\def\l{\left(}
\def\r{\right)}
\def\p{\mathbb{P}}
\def\e{\mathbb{E}}

\def\1{\mathds{1}}

\def\({\left(}
\def\){\right)}

\def\l{\left(}
\def\r{\right)}
\def\p{\mathbb{P}}
\def\e{\mathbb{E}}

\begin{document}
\title[Cram\'er's large deviation for BPREs]{Berry-Esseen's bound and Cram\'er's large deviation expansion
for a supercritical branching process in a random environment}

\author{Ion~Grama}
\curraddr[Grama, I.]{ Universit\'{e} de Bretagne-Sud, LMBA, UMR CNRS 6205,
Vannes, France}
\email{ion.grama@univ-ubs.fr}

\author{Quansheng Liu}
\curraddr[Liu, Q.]{ Universit\'{e} de Bretagne-Sud, LMBA, UMR CNRS 6205,
Vannes, France}
\email{quansheng.liu@univ-ubs.fr}

\author{Eric Miqueu}
\curraddr[Miqueu, E.]{Universit\'{e} de Bretagne-Sud, LMBA, UMR CNRS 6205,
Vannes, France}
\email{eric.miqueu@univ-ubs.fr}
\date{\today }

\begin{abstract}
Let $(Z_n)$ be a supercritical branching process in a random environment $\xi = (\xi_n)$. 
We establish a Berry-Esseen bound and a Cram\'er's type large deviation expansion
for $\log Z_n$ under the annealed law $\mathbb P$.
We also improve some earlier results  about the harmonic moments of the limit variable $W=lim_{n\to \infty} W_n$, where $W_n =Z_n/ \mathbb{E}_{\xi} Z_n$ 
is the normalized population size.
\end{abstract}

\maketitle

\section{Introduction and main results}

A branching process in a random environment (BPRE) is a natural and important generalisation
of the Galton-Watson process, where the reproduction law varies according to a random environment indexed by time.
It was introduced for the first time in Smith and Wilkinson \cite%
{smith} to modelize the growth of a population submitted to an environment.
For background concepts and basic results concerning a BPRE we refer to 
Athreya and Karlin \cite{athreya1971branching, athreya1971branching2}.
In the critical and subcritical regime the process goes out and the research interest is concentrated mostly on the survival probability
and conditional limit theorems for the branching process,
see e.g. Afanasyev, B\"oinghoff, Kersting and Vatutin \cite{afanasyev2012limit, afanasyev2014conditional}, Vatutin \cite{Va2010},  Vatutin and Zheng \cite{VaZheng2012}, and the references therein.
In the supercritical
case, a great deal of current research has been focused on large deviation principle, 
see Bansaye and Berestycki \cite{bansaye2009large}, B\"oinghoff and Kersting \cite{boinghoff2010upper}, Bansaye and B\"oinghoff \cite{bansaye2011upper, bansaye2013lower, bansaye2014small}, Huang and Liu \cite{liu}. 
In the particular case when the offspring distribution is geometric, precise asymptotics can be found in 
Kozlov \cite{kozlov2006large}, B\"oinghoff \cite{boinghoff2014limit}, Nakashima \cite{Nakashima2013lower}. 
In this article, we complete on these results by giving the Berry-Esseen bound 
and asymptotics of large deviations of Cram\'{e}r's type for a supercritical BPRE.

A BPRE can be described as follows.
 The random environment is
represented  by a sequence $\xi = (\xi_0, \xi_1 , ... ) $ of independent and
identically distributed random variables (i.i.d.\ r.v.'s); each realization of $\xi_n$ corresponds to a probability law $\{ p_i(\xi_n): i \in \mathbb N \}$ on 
$\mathbb N = \{0,1,2,\dots\},$
whose probability generating function is
\begin{equation}
\mathnormal{f}_{\xi_n} (s)= f_n (s) = \sum_{i=0}^{\infty} p_i ( \xi_n ) s^i,
\quad s \in [0,1], \quad p_i ( \xi_n ) \geq 0, \quad \sum_{i=0}^{ \infty}
p_i (\xi_n) =1.  \label{defin001}
\end{equation}
Define the process $(Z_n)_{n \geq 0}$ by the relations
\begin{equation}  \label{relation recurrence Zn}
Z_0 = 1, \quad Z_{n+1} = \sum_{i=1}^{Z_n} N_{n, i}, \quad \text{for} \quad n
\geq 0,
\end{equation}
where $N_{n,i} $ is the number of children of the $i$-th individual of the
generation $n$. Conditionally on the environment $\xi $, the r.v.'s $N_{n,i} $
(i = 1, 2, ...) are independent of each other with common probability generating function $\mathnormal{f}_n,$
and also independent of $Z_n$. 

In the sequel we denote by $\mathbb{P}_{\xi}$ the
\textit{quenched law}, i.e.\ the conditional probability when the
environment $\xi$ is given, and by $\tau $ the law of the environment $\xi$.
Then
$\mathbb{P}(dx,d\xi) = \mathbb{P}_{\xi}(dx) {\tau}(d\xi)$
is the total law of the process, called
\textit{annealed law}. The corresponding quenched and annealed expectations
are denoted respectively by $\mathbb{E}_{\xi}$ and $\mathbb{E}$. We also
define, for $n\geq 0$,
\begin{equation*}
m_n = m_n ( \xi )= \sum_{i=0}^\infty i p_i ( \xi_n ) \quad \text{and} \ \ \Pi_n = \mathbb{E}_{\xi} Z_n = m_0 ... m_{n-1},
\end{equation*}
where $m_n $ represents the average number of children of an
individual of generation $n$  when the environment $\xi $ is given. Let
\begin{equation} \label{Wn}
W_n =\frac{Z_n}{\Pi_n} , \quad n\geq 0,
\end{equation}
be
the normalized population size.
It is well known that under $%
\mathbb{P}_{\xi}$, $(W_n)_{n \geq 0} $ is a non-negative martingale with respect to the filtration
$$\mathcal{F}_n = \sigma
\left(\xi, N_{k,i} , 0 \leq k \leq n-1, i = 1,2 \ldots \right), $$
where by convention $\mathcal{F}_0 = \sigma(\xi)$.
Then the limit $W = \lim W_n $ exists $\mathbb{P}$ - a.s. and $\mathbb{E} W \leq 1 $.

An important tool in the study  of a BPRE  is the associated random
walk
\begin{equation*}
S_n = \log \Pi_n = \sum_{i=1}^{n} X_i , \quad n \geq 1, 
\end{equation*}
where the r.v.'s  $X_i = \log m_{i-1}$ $(i\geq1)$ are i.i.d.\  depending only on the environment $\xi$.
It turns out that the behavior  of the process $(Z_n)$ is mainly determined by the associated random walk which is seen from the decomposition
\begin{equation} \label{decom-logZn}
\log Z_{n}=S_{n}+\log W_{n}.
\end{equation}
For the sake of brevity set $X=  \log m_0 $,
\begin{equation*}
\mu = \mathbb{E} X \quad \text{and} \quad \sigma^2 = \mathbb{E} (X - \mu)^2.
\end{equation*}
We shall assume that the BPRE is supercritical, with $\mu \in (0, \infty );$
together with $\mathbb E | \log (1-p_0(\xi_0)) | < \infty$ this implies that 
 the population size tends to infinity with positive probability  (see \cite{athreya1971branching}).
We also assume that the random walk $(S_n)$ is non-degenerate with $0<\sigma^2<\infty$;
in particular this implies that
\begin{equation}
\label{EP_1}
\mathbb P(Z_1=1)= \mathbb E p_1(\xi_0) <1.
\end{equation}

Throughout the paper, we assume the following condition:
\begin{equation}
\mathbb{E} \frac{Z_1 \log^+ Z_1}{m_0}  < \infty , \label{A0}
\end{equation}
which implies that the martingale $W_n$  converges to $W$ in $L^1 (\mathbb{P})$ (see e.g. \cite{tanny1988necessary}) and
\[ \mathbb{P} (W>0) =  \mathbb{P} (Z_n \to \infty) = \lim_{n \to \infty} \mathbb{P} (Z_n >0) >0. \]
Furthermore, we assume in the sequel that each individual has at least one child,
which means that
\begin{equation}
p_0 = 0 \quad \mathbb{P} \mbox{ - a.s.}  \label{p0}
\end{equation}
In particular this implies that the associated random walk has positive increments, $Z_n \to \infty$ and $W>0$ $\mathbb{P}$ - a.s.
Throughout the paper, we denote by $C$ an absolute constant whose value may differ
from line to line.

Our first result is a Berry-Esseen type bound for $\log Z_n,$ which
holds under the following additional assumptions:
\vskip0.2cm \noindent
\textbf{A1.} 
There exists a constant $\varepsilon >0$ such that
\begin{equation}
\mathbb{E} X^{3+ \varepsilon} < \infty. \label{Moment3}
\end{equation}
\textbf{A2.} There exists a constant $p>1$ such that
\begin{equation}
\mathbb{E} \left( \frac{Z_1}{m_0} \right)^p < \infty. \label{H2 bis}
\end{equation}
\begin{theorem}
\label{BE theorem for BPRE} Under conditions \textbf{A1} and \textbf{A2}, we have
\begin{equation*}
\sup_{x \in \mathbb{R} } \left| \mathbb{P}\left( \frac{\log Z_n - n \mu}{%
\sigma \sqrt{n}} \leq x \right) - \Phi (x) \right| \leq \frac{C}{\sqrt{n}},
\end{equation*}
where $ \displaystyle{\ \Phi (x) = \frac{1}{\sqrt{2 \pi}} \int_{- \infty}^{x} e^{-
t^2 / 2 } dt }$ is the standard normal distribution function.
\end{theorem}
Theorem \ref{BE theorem for BPRE} completes the results of \cite{liu} by giving the rate of convergence in the central limit theorem for $\log Z_n$. The proof of this theorem is based on  Stein's method and is deferred to
Section \ref{section stein}.

Our next result concerns the asymptotic behavior of the left-tail of the r.v.\ $W$. 
For the Galton-Watson process  this problem is well studied, 
see e.g. \cite{fleischmann} and the references therein. 
For a BPRE, some interesting  results have been obtained in
 \cite{hambly} and \cite{liu}. In particular, for the annealed law,
Huang and Liu (\cite{liu}, Theorem 1.4) have found a necessary and sufficient condition
for the existence of harmonic moments of $W$, under the following
hypothesis:
\begin{equation*}
(H) \quad \exists \ \delta > 0 \ \text{and} \ A > A_1 >1 \ \text{such that}
\ \ A_1 \leq m_0 \ \ \text{and} \ \sum_{i=1}^\infty i^{1+ \delta} p_i (\xi_0)  \leq A^{1+ \delta} \ a.s.
\end{equation*}
However, this hypothesis is very restrictive; it implies in particular that $1<A_1\leq m_0\leq A$. 
We will show (see Theorem \ref{harmonic moment} below)  the existence of harmonic moments
under the following significantly less restrictive assumption:
\vskip0.2cm \noindent \textbf{A3.}
The r.v.\ $X = \log m_0$ has an exponential moment, i.e. there exists a constant $\lambda_0 >0$ such
that
\begin{equation}
\mathbb{E} e^{ \lambda_0 X } = \mathbb{E}m_0^{\lambda_0} <\infty.
\label{MomentExp}
\end{equation}
Under this hypothesis,  since $X$ is a positive random variable, the function $ \lambda \mapsto \mathbb{E}e^{\lambda X}$ is finite for all $\lambda \in (-\infty, \lambda_0]$ and is increasing.
\begin{theorem}
\label{harmonic moment} Assume condition \textbf{A3.}  Let
\begin{equation}
a_0 = \left\{
\begin{array}{cl}
 \frac{ \lambda_0}{1-{ \log \mathbb{E} m_0^{\lambda_0}  } / {\log \mathbb{E} p_1}   } & \text{if} \quad \mathbb P ( p_1>0 )  >0 , \\
\lambda_0 & \text{otherwise}. \\
\end{array}%
\right.
\label{Def a0}
\end{equation}%
Then, for all $a \in (0,a_0)$,
\begin{equation*}
\mathbb{E} W^{- a } < \infty .
\end{equation*}
\end{theorem}
Yet, a necessary and sufficient condition for the existence of harmonic moments of order $a > 0$ under condition \textbf{A3} is still an open question.

The previous theorem
allows us to obtain a Cram\'er type large deviation expansion for a BPRE.
To state the corresponding result we need more notations. Let $L$ and $\psi$ be respectively the
moment and cumulant generating function of the random variable $X $:
\begin{equation}
L ( \lambda ) =
\mathbb{E} e^{\lambda X } = \mathbb{E} \left( m_0^{\lambda} \right) ,
\end{equation}
\begin{equation} \label{def-phi}
\psi (\lambda) = \log L (\lambda).
\end{equation}
Then $\psi$ is analytical for $\lambda \leq \lambda_0$ and we have $\psi (\lambda) = \sum_{k=1}^{\infty} \frac{\gamma_k%
}{k !} \lambda^k, $ where $\gamma_k = \frac{d^{k} \psi }{d \lambda^k} (0) $ is the
cumulant of order $k$ of the random variable $X$. In particular for $%
k =1,2$, we have $\gamma_1 =  \mu $ and $\gamma_2 = \sigma^2 $.
We shall use the Cram\'er's series of the associated random walk $(S_n)_{n
\geq 0}$ defined by
\begin{equation}  \label{Cram series}
\mathscr L (t) = \frac{\gamma_3}{6 \gamma_2^{3/2}} + \frac{\gamma_4 \gamma_2
- 3 \gamma_3^2}{24 \gamma_2^3} t + \frac{\gamma_5 \gamma_2^2- 10 \gamma_4
\gamma_3 \gamma_2 + 15 \gamma_3^3}{120 \gamma_2^{9/2}} t^2 + \ldots
\end{equation}
(see Petrov \cite{petrov}) which converges for $|t|$ small enough. 

Consider the following assumption:
\vskip0.2cm \noindent \textbf{A4.}
There exists a constant $p>1$ such that
\begin{equation}
\mathbb{E}  \frac{Z_1^{p}}{m_0}  < \infty.
\end{equation}
\vskip0.2cm
Note that under \eqref{p0} condition \textbf{A4} implies  \textbf{A2}. The intuitive meaning of these conditions is that the process $(Z_n)$ cannot deviate too much from its mean $\Pi_n$.

The following theorem 
gives a Cram\'{e}r's type large deviation expansion of a BPRE.
\begin{theorem}
\label{cramer type theorem}
Assume conditions \textbf{A3} and \textbf{A4.} Then, for $0 \leq x = o(\sqrt{n})$, we have, as $n \to \infty,$
\begin{equation}
\frac{\mathbb{P} \left( \frac{\log Z_n- n \mu}{\sigma \sqrt{n}}> x \right)}{%
1- \Phi (x)} = \exp \left\{ \frac{x^3}{\sqrt{n}} \ \mathscr L \left( \frac{x}{\sqrt{n}} \right) \right\} \left[ 1+O \left( \frac{1+x}{\sqrt{n}}
\right) \right]
\end{equation}
and
\begin{equation}
\frac{\mathbb{P} \left( \frac{\log Z_n - n \mu}{\sigma \sqrt{n}}< - x \right)%
}{\Phi (-x)} = \exp \left\{- \frac{x^3}{\sqrt{n}} \ \mathscr L \left(- \frac{x}{\sqrt{n}} \right) \right\} \left[ 1+O \left( \frac{1+x}{\sqrt{n}}
\right) \right].
\end{equation}
\end{theorem}
As a consequence of this result we obtain a large deviation approximation by the normal law in the  normal zone $x=o(n^{1/6}):$ 
\begin{corollary}
\label{theorem normal zonz} Under the assumptions of Theorem \ref{cramer type theorem},
 we have for $0 \leq x = o(n^{1/6}) $, as $n \to \infty,$
\begin{equation}
\frac{\mathbb{P} \left( \frac{\log Z_n- n \mu}{\sigma \sqrt{n}}> x \right)}{1- \Phi (x)} = 1 + O \left( \frac{x^3}{\sqrt{n}} \right)
\end{equation}
and
\begin{equation}
\frac{\mathbb{P} \left( \frac{\log Z_n - n \mu}{\sigma \sqrt{n}}< - x \right)%
}{\Phi (-x)} = 1 + O \left( \frac{x^3}{\sqrt{n}} \right) .
\end{equation}
\end{corollary}

Note that Theorem \ref{cramer type theorem}  is more precise than the moderate deviation principle established in \cite{liu},  and, moreover, 
is stated under weaker assumptions.
Indeed, let $a_n$ be a sequence of positive numbers satisfying $\frac{a_n}{n} \to 0$ and $\frac{a_n}{\sqrt{n}} \to \infty$. Then by
Theorem 1.6 of \cite{liu},  under hypothesis $(H)$, we have,
for  $x_n = \frac{x a_n}{\sigma \sqrt{n}}$ with fixed $x \in \mathbb{R}$,
\begin{equation}
\label{MDP}
  \log \bb{P} \l \frac{\log Z_n - n \mu }{a_n} > x \r \sim   - \frac{x_n^2}{2}.
\end{equation}
Using the weaker condition \textbf{A3} (instead of condition $(H)$) 
Theorem \ref{cramer type theorem} implies that 
\begin{equation}
\label{CT}
\bb{P} \l \frac{\log Z_n - n \mu }{\sigma \sqrt{n}} > x_n \r =  \l 1 - \Phi (x_n) \r \exp \l \frac{x_n^3}{\sqrt{n}} \mathscr L  \l \frac{x_n}{\sqrt{n}} \r \r \l 1 + O \l \frac{1+x_n}{\sqrt{n}} \r \r ,
\end{equation}
which sharpens  (\ref{MDP}) without the log-scaling.

The rest of the paper is organized as follows. 
In Section 2, we prove
Theorem \ref{BE theorem for BPRE}. 
In Section 3, we study the
existence of harmonic moments of  $W$
and give a proof of Theorem \ref{harmonic moment}. 
Section 4 is devoted to the proof of Theorem \ref{cramer type theorem}.

\section{The Berry-Essen bound for $\log Z_n $}
\label{section stein}
In this section we establish a Berry-Esseen bound for the
normalized branching process
\begin{equation*}
\frac{\log Z_n - n \mu}{\sigma \sqrt{n}},
\end{equation*}
based on Stein's method. 
In Section \ref{SecBE1}, we recall briefly the main
idea of Stein's method. 
Section \ref{SecBE2} contains some auxiliary results to be used latter in the proofs. 
In Section \ref{SecBE3}, we give a proof of Theorem \ref{BE theorem for BPRE}.

\subsection{Stein's method}
\label{SecBE1}
Let us recall briefly some facts on the Stein method to be used in the proofs. 
For more details, the reader can consult the excellent reviews \cite{barbour_steins_2014, ross} or the more complete book \cite{barbour}. 
The main idea is to describe the closeness of the law of a r.v.\ $X$ to the standard normal law using
\textit{Stein's operator}
\begin{equation}
\mathcal{A} \mathnormal{f} (w) = \mathnormal{f}^{\prime }(w) - w \mathnormal{%
f} (w),
\end{equation}
which can be seen as a substitute of the classical Fourier-transform tool.
For any $x \in \mathbb{R}$ let $f_x$ be a solution of  \textit{Stein's equation} :
\begin{equation}
\label{stein equation}
\mathds{1}( w \leq x ) - \Phi (x) = f_x^{\prime }(w) - w f_x (w) ,
\end{equation}
for all $w \in \mathbb{R}$.
The Kolmogorov distance between the law of the random variable $X$ and the normal law ${\mathcal N}(0,1)$ can be expressed in term of {\it Stein's expectation} $\mathbb{E} \mathcal{A} f_x (X)$.
Indeed, substituting $w$ by $X$ in (\ref{stein equation}), taking expectation and the supremum over $x \in \mathbb{R}$, we obtain
\begin{equation}
\label{distance kolmorogov stein}
\sup_{x \in \mathbb{R}} \left| \mathbb{P} \left( X \leq x \right) - \Phi (x)
\right| = \sup_{x \in \mathbb{R}}\left| \mathbb{E} \left(f_x (X) - X f_x (X)\right) \right| = \mathbb{E} \mathcal{A} f_x (X).
\end{equation}
The key point is that  Stein's operator $\mathcal{A}$ characterizes the standard normal law,  as  shown by
the following Lemma.
\begin{lemma}[Characterization of the normal law] 
\label{lem stein Characterization of the normal law}
A random variable $Z$ is of normal law $ \mathcal{N} (0,1)$ if and only if \ 
$\mathbb{E} \mathcal{A} \mathnormal{f} (Z) = 0$ 
for all absolutely continuous function  $f$  such that
$\mathbb{E} | \mathnormal{f}^{\prime }(Z) | < \infty.$
\end{lemma}
By Lemma \ref{lem stein Characterization of the normal law}, it is expected
that if the distribution of $X$ is close to the normal law $ \mathcal{N} (0,1)$ in the sense of Kolmogorov's distance, then  $\mathbb{E}
\mathcal{A} \mathnormal{f} (X)$ is close to $0$ for a large class of
functions $f$ including  the solutions $f_x$ of Stein's equation \eqref{stein equation}.  This
permits to study the convergence of $X$ to the normal law  by using only the structure of $X$ and the qualitative properties of  $f_x$.  
We will use the following result, where we use the notation $\|\cdot \|$ for the infinity norm. 
\begin{lemma}
\label{lem stein fh} 
For each $x\in {\mathbb R}$,  Stein's equation (\ref{stein equation}) has a unique bounded solution (see \cite{chenhoLp}, Lemma 1.1)
given by
\begin{eqnarray}
\mathnormal{f}_x (w) &=& e^{w^2/2} \int_w^{\infty} e^{-t^2/2} ( \Phi (x) - \mathds{1} (t \leq x)) dt \\
&=&  \left\{ \begin{tabular}{l l}
            $\sqrt{2 \pi } e^{w^2/2} \Phi (w) \left[ 1 - \Phi (x) \right]$ & if $ w \leq x$, \\
            $\sqrt{2 \pi } e^{w^2/2} \Phi (x) \left[ 1 - \Phi (w) \right]$ & if $ w  > x$.
            \end{tabular} \right. \nonumber
\end{eqnarray}
Moreover, we have for all real $x$,
\begin{equation}
\| f_x \| \leq 1, \quad \| f_x' \| \leq 1,
 \end{equation}
and for all real $w$, $s$ and $t$ (see \cite{chenhoLp}, Lemma 1.3),
\begin{eqnarray}  \label{majoration fx'}
\left| f_x^{\prime }(w+s) -f_x^{\prime }(w+t) \right| &\leq& ( |t| +
|s|)(|w|+1) + \mathds{1}(x-t \leq w \leq x-s) \mathds{1} (s \leq t)  \notag
\\
&&+ \mathds{1}(x-s \leq w \leq x-t) \mathds{1} (s>t).
\end{eqnarray}
\end{lemma}
The next result gives a bound of order $n^{-1/2}$ of Stein's expectation of a sum of i.i.d.\ r.v.'s.
\begin{lemma}
\label{lem borne eq stein}
Let $X_1, \ldots , X_n$ be a sequence of i.i.d.\ r.v.'s with $\mu = \e X_1 \in \mathbb R$, $\sigma^2 = \e \left[ X_1 - \mu \right]^2 < \infty$ and $\rho = \e |X_1|^3 < \infty$. Define $Y_n = \frac{1}{\sigma \sqrt{n} } \sum_{k=1}^{n} (X_k- \mu) $. For each $x\in {\mathbb R}$,  
the unique bounded solution $f_x$ of  Stein's equation (\ref{stein equation}) satisfies
\begin{equation}
\label{equation de stein iid}
\left| \mathbb{E} \left[ f_x^{\prime }(Y_n) - Y_n f_x (Y_n ) \right] \right|
\leq C \rho / \sqrt{n},
\end{equation}
where $C$ is an absolute constant.
\end{lemma}
Note that from \eqref{distance kolmorogov stein} and \eqref{equation de stein iid} one gets the classical Berry-Esseen theorem. The proof of Lemma \ref{lem borne eq stein} can be found in \cite{chenhoLp}. 

\subsection{Auxiliary results}
\label{SecBE2}
In the proof of Theorem \ref{BE theorem for BPRE} we make use of the following two assertions. 
The first one is a consequence of the Marcinkiewicz-Zygmund inequality (see \cite{liu2001local}, Lemma 1.4), which  will  be used several times.
\begin{lemma}[\cite{liu2001local}, Lemma 1.4] \label{lemma MZ}
Let $(X_i)_{i \geq 1}$ be a sequence of i.i.d.\ centered r.v.'s. Then we have for $p\in (1, \infty)$,
\begin{equation} \label{MZ inequality}
\mathbb{E} \left| \sum_{i=1}^{n} X_i \right|^p \leq
\left\{
\begin{array}{ll}
(B_p)^p \mathbb{E} \left( | X_i |^p \right) n , & \text{if} \  \ 1 < p \leq 2, \\
(B_p)^p \mathbb{E} \left( | X_i |^p \right) n^{p/2} , & \text{if }\ \  p>2,%
\end{array}%
\right.
\end{equation}
where $B_p = 2 \min \left\{ k^{1/2} : k \in \mathbb{N}, k \geq p/2 \right\}$ is a constant depending only on  $p$  (so that $B_p =2 $ if $1<p \leq 2$).
\end{lemma}
The second one  is a result concerning the exponential rate of convergence 
of $W_n$ to $W$
in $L^p ( \mathbb{P} )$ from \cite{huang_convergence_2014}, Theorem 1.5.
\begin{lemma} \label{convergence Lp exponentielle} 
Under  \textbf{A2},  there exist two  constants $C>0$ and $\delta \in (0,1)$ such that
\begin{eqnarray*}
\left(\mathbb{E} \left| W_n-W \right|^p \right)^{1/p} \leq C \delta^n .
\end{eqnarray*}
\end{lemma}
The next result concerns the existence of positive moments of the r.v. $\log W$.
\begin{lemma}
\label{moment log W} 
Assume that $\mathbb{E} | \log m_0 |^{2p} < \infty$, for some  $%
p>1.$ Then we have, for all $q \in (0, p)$, 
\begin{equation*} 
\mathbb{E} | \log W|^{q} < \infty  \quad \text{and} \quad \underset{n \in \mathbb{N}}{\sup} \ \mathbb{E} |\log W_n |^{q} < \infty . 
\end{equation*}
\end{lemma}
We prove Lemma \ref{moment log W} by studying the asymptotic  behavior of the Laplace transform of $W$.
Define the quenched  and annealed Laplace transform of $W$ by
\begin{equation*}
\phi_{\xi} (t) = \mathbb{E}_{\xi} e^{-t W} \ \ \ \text{and} \ \ \ \phi (t) =
\mathbb{E} \phi_{\xi} (t) = \mathbb{E} e^{-t W},
\end{equation*}
where $t \geq 0$.
Then by Markov's inequality, we have for $t>0$,
\begin{equation}
\label{tauberian theorem}
\mathbb{P} ( W < t^{-1} ) \leq e \; \mathbb{E} e^{- t W} = e \;  \phi (t).
\end{equation}
\begin{proof}[Proof of Lemma \ref{moment log W}]
By H\"older's inequality, it is enough to prove the assertion of the lemma for $q\in (1,p).$
It is obvious that there exists a constant $C>0$ such that $\mathbb{E} | \log W |^{q} \mathds{1}(W \geq 1 ) \leq C \bb{E} W < \infty $. So
it remains to show that $\mathbb{E} | \log W |^{q} \mathds{1}( W \leq 1 ) < \infty $. 
By \eqref{tauberian theorem} and the fact that
\begin{equation}
\label{eq expression integrale E log W^q}
\e \left| \log W \right|^{q} \mathds{1} ( W \leq 1) = q \int_{1}^{+ \infty} \frac{1}{t} \left( \log t  \right)^{q-1} \p ( W \leq t^{-1} ) dt,
\end{equation}
it is enough to show that, as $t \to \infty,$
\begin{equation*}
\phi (t) =  O ( \log t)^{-p} .
\end{equation*}
It is well-known that $\phi_{\xi} (t) $ satisfies the functional relation
\begin{equation} \label{2.5}
\phi _{\xi} (t) = \mathnormal{f}_0 \left( \phi_{T \xi} \left( \frac{t}{m_0} \right) \right),
\end{equation}
where $f_0$ is the generating function defined by (\ref{defin001}) 
and $T^n$ is the shift operator defined by  $T^n(\xi_0, \xi_1, \ldots ) = (\xi_n, \xi_{n+1}, \ldots )$ for $n\geq 1.$
Using (\ref{2.5}) and the fact that $ \phi^k_{T \xi} \left( \frac{t}{m_0} \right) \leq \phi^2_{T \xi} \left( \frac{t}{m_0} \right)$ for all $k \geq 2$, we obtain
\begin{eqnarray} \label{2.5a} \nonumber
\phi_{\xi} (t) &\leq& p_1 (\xi_0) \phi_{T \xi } \left( \frac{t}{m_0} \right) + ( 1 - p_1 ( \xi_0 ) ) \phi_{T \xi }^2 \left( \frac{t}{m_0} \right) \\
&=& \phi_{T \xi } \left( \frac{t}{m_0 } \right) \left( p_1 ( \xi_0 ) + ( 1- p_1 ( \xi_0 ) ) \phi_{T \xi } \left( \frac{t}{m_0} \right) \right).
\end{eqnarray}
By iteration, this leads to
\begin{equation} \label{2.6}
\phi_{\xi} (t) \leq \phi_{T^n \xi} \left(\frac{t}{\Pi_n}\right) \ \prod_{j=0}^{n-1} \left( p_1 ( \xi_j) + (1- p_1 ( \xi_j) ) \phi_{T^n \xi} \left( \frac{t}{\Pi_n} \right) \right).
\end{equation}
Taking expectation and using the fact that  $ \phi_{T^n \xi} (t)  \leq 1$, we get
\[ \phi (t) \leq \mathbb{E} \left[ \prod_{j=0}^{n-1} \left( p_1 ( \xi_j) + (1- p_1 ( \xi_j) ) \phi_{T^n \xi} \left( \frac{t}{\Pi_n} \right) \right) \right]. \]
Using a simple truncation and the fact that $ \phi_{\xi} ( \cdot ) $ is non-increasing, we have, for all $A>1$,
\begin{eqnarray*}
\phi (t) &\leq& \mathbb{E} \left[  \prod_{j=0}^{n-1} \left( p_1 ( \xi_j) + (1- p_1 ( \xi_j) ) \phi_{T^n \xi} \left( \frac{t}{A^n} \right) \right) \mathds{1}( \Pi_n \leq A^n ) \right] + \mathbb{P}( \Pi_n \geq A^n ) \\
&\leq& \mathbb{E} \left[  \prod_{j=0}^{n-1} \left( p_1 ( \xi_j) + (1- p_1 ( \xi_j) ) \phi_{T^n \xi} \left( \frac{t}{A^n} \right) \right)  \right] + \mathbb{P}( \Pi_n \geq A^n ).
\end{eqnarray*}
Since $T^n \xi$ is independent of $\sigma ( \xi_0, ... ,\xi_{n-1} )$, and the r.v.'s $ p_1 ( \xi_i)$  ($i\geq 0$) are i.i.d., we have
\begin{equation*}
\phi (t)  \leq \left[ \mathbb{E} p_1 (\xi_0) + (1-  \mathbb{E} p_1 ( \xi_0 ) )  \phi \left(\frac{t}{A^n}\right) \right]^n + \mathbb{P}( \Pi_n \geq A^n ).
\end{equation*}
By the dominated convergence theorem, we have $\lim_{t \to \infty} \phi (t) = 0$. Thus, for any $\gamma \in (0,1)$, there exists a constant $K>0$ such that, for all $ t \geq K$, we have $\phi(t) \leq \gamma$.
Then for all
$ t \geq K A^n,$ we have $ \phi  \left(\frac{t}{A^n}\right) \leq \gamma $. 
Consequently, for $ t  \geq K A^n$,
\begin{equation}
\phi (t) \leq \alpha^n + \mathbb{P}( \Pi_n \geq A^n ) ,
\label{alpha0}
\end{equation}
where, by (\ref{EP_1}), 
\begin{equation}
\alpha = \mathbb{E} p_1 (\xi_0) + (1-  \mathbb{E} p_1 (\xi_0 ) ) \gamma  \in (0,1).
\label{alpha}
\end{equation}
Recall that $\mu= \mathbb{E} X$ and $S_n=\log \Pi_n = \sum_{i=1}^{n} X_i$. Choose $A$ such that $\log A>\mu$ and let $\delta =  \log A-\mu >0$. By Markov's inequality and Lemma \ref{lemma MZ}, there exists a constant $C>0$ such that, for $n \in \mathbb{N}$,
\begin{eqnarray*}
\mathbb{P}( \Pi_n \geq A^n ) &\leq& \mathbb{P} \left( \left| S_n - n \mu \right| \geq n \delta \right) \\
&\leq& \frac{\mathbb{E} \left| \sum_{i=1}^{n} (X_i- \mu) \right|^{2p} }{n^{2p} \delta^{2p} } \\
&\leq& \frac{C}{n^{p}}.
\end{eqnarray*}
Then, by (\ref{alpha0}), we get, for $n$ large enough and $t \geq  K A^n$,
\begin{equation} \label{18}
\phi (t) \leq \frac{C}{n^{p}}.
\end{equation}
For $t \geq K,$ define $n_0 = n_0 (t) = \left[ \frac{ \log (t/K)}{ \log (A)} \right] \geq 0,$ where $ [x]$ stands for the integer part of $x$, so that
\[
\frac{\log (t/K)}{\log(A)} - 1 \leq n_0 \leq \frac{\log (t/K)}{\log(A)}
\quad \text{and} \quad  t \geq K A^{n_0}.
\]
Coming back to (\ref{18}), with  $n=n_0$, we get for $ t \geq K $,
\[ \phi (t) \leq \frac{ C ( \log A)^{p}}{ ( \log (t/K) )^{p} } \leq C ( \log t )^{-p}, \]
which proves that $ \mathbb{E} | \log W |^{q} < \infty$ for all $q \in (1,p)$, (see \eqref{eq expression integrale E log W^q}).
Furthermore, since $x \mapsto |\log^{q}  (x) | \mathds{1}( x \leq 1) $ is a non-negative and convex function for $q\in (1,p)$,  by Lemma 2.1 of \cite{liu} we have
$$
\underset{n \in \mathbb{N}}{\sup} \ \mathbb{E} \left| \log W_n \right|^{q}  \mathds{1}( W_n \leq 1 )  = \mathbb{E} \left| \log W \right|^{q}  \mathds{1}( W \leq 1 ) .
$$
By a standard truncation we obtain
\begin{equation}
\label{majoration uniforme moment log W}
\underset{n \in \mathbb{N}}{\sup} \ \mathbb{E} \left| \log W_n \right|^{q}  \leq  C \mathbb{E}W + \mathbb{E} \left| \log W \right|^{q}  \mathds{1}( W \leq 1 )< \infty ,
\end{equation}
which ends the proof of the lemma.
\end{proof}
The next result concerns the exponential speed of convergence of $\log W_n$ to $\log W$.
\begin{lemma}
\label{lemma logW cv exp}
Assume \textbf{A2} and  there exists a constant $q>2$ such that $\mathbb{E}| \log m_0|^{q} < \infty$. Then there exist two  constants $C>0$ and $\delta \in (0,1)$ such that for all $n\geq 0$,
\begin{equation}
 \mathbb{E} \left| \log W_n -\log W \right| \leq C \delta^n.
\label{exp conv log W 2}
\end{equation}
\end{lemma}
\begin{proof}
From (\ref{relation recurrence Zn}) and (\ref{Wn}) we get the following useful decomposition:
\begin{equation}
\label{difference Wn+1-Wn}
W_{n+1} - W_n = \frac{1}{\Pi_n} \sum_{i=1}^{Z_n} \left( \frac{N_{n,i}}{m_n}-1 \right),
\end{equation}
which reads also
\begin{equation}
\label{eta_n}
\frac{W_{n+1}}{W_n}-1 = \frac{1}{Z_n} \sum_{i=1}^{Z_n} \left( \frac{N_{n,i}}{m_n}-1 \right).
\end{equation}
By (\ref{eta_n}) we have the decomposition
\begin{equation}
\label{decomposition log Wn eta_n}
\log W_{n+1} - \log W_n = \log ( 1 + \eta_{n} ),
\end{equation}
with
\begin{equation}
\eta_{n} = \frac{W_{n+1}}{W_n}-1= \frac{1}{Z_n} \sum_{i=1}^{Z_n} \left( \frac{N_{n,i}}{m_n} -1 \right).
\end{equation}
Under ${\mathbb P}_\xi$  the r.v.'s $\frac{N_{n,i}}{m_n}-1 \; (i \geq 1) $ are i.i.d., centered and independent of $Z_n$.

Choose $p \in (1,2]$ such that \textbf{A2} holds.
We first show that
\begin{equation}
\label{MZ eta_n}
\left(\mathbb{E}|\eta_n|^p \right)^{1/p} \leq C \delta^{n},
\end{equation}
for some constants $C>0$ and $\delta \in (0,1)$.
Applying Lemma \ref{lemma MZ} under $\mathbb{P}_{\xi}$ and using the independence between the r.v.'s $\frac{N_{n,i}}{m_n} \; (i \geq 1)$ and $Z_n$, we get
\begin{eqnarray*}
\mathbb{E}_{\xi} |\eta_n|^p &\leq& 2^p \mathbb{E}_{\xi} \left[ Z_n^{1-p} \right] \mathbb{E}_{\xi} \left| \frac{N_{n,1}}{m_n}-1 \right|^p .
\end{eqnarray*}
By \textbf{A2} and the fact that under the probability $\mathbb {P}$ the random variable $\frac{N_{n,1}}{m_n}$ has the same law as $ \frac{Z_1}{m_0}$, we obtain
\begin{equation}
\label{majoration eta_n}
\mathbb{E}|\eta_n|^p  \leq 2^p \mathbb{E} \left| \frac{Z_1}{m_0} -1 \right|^p \mathbb{E}\left[Z_n^{1-p} \right] .
\end{equation}
We shall give a bound of the harmonic moment $\mathbb{E}Z_n^{1-p}$. By (\ref{relation recurrence Zn}), using the convexity of the function $x \mapsto x^{1-p}$ and the independence between the r.v.'s $Z_n$ and $N_{n,i} \ ( i \geq 1)$, we get
\begin{eqnarray*}
\mathbb{E}\left[Z_{n+1}^{1-p}\right]&=&\mathbb{E}\left[ \left(\sum_{i=1}^{Z_{n}} N_{n,i}\right)^{1-p}\right] \nonumber \\
&\leq& \mathbb{E} \left[ Z_n^{1-p} \frac{1}{Z_n} \left(\sum_{i=1}^{Z_{n}} N_{n,i}^{1-p} \right) \right] \nonumber\\
&\leq& \mathbb{E} \left[ \mathbb{E} \left( Z_n^{1-p} \frac{1}{Z_n} \left(\sum_{i=1}^{Z_{n}} N_{n,i}^{1-p} \right)  \Bigg| Z_n \right) \right] \nonumber\\
&=& \mathbb{E} \left[ Z_n^{1-p}\right] \mathbb{E} \left[ N_{n,1}^{1-p}\right] \nonumber\\
&=& \mathbb{E} \left[ Z_n^{1-p}\right] \mathbb{E} \left[ Z_1^{1-p}\right].
\end{eqnarray*}
By induction, we obtain
$$
\mathbb{E}\left[Z_{n+1}^{1-p}\right] \leq \left(\mathbb{E}  Z_1^{1-p} \right)^{n+1}. $$
By (\ref{p0}), we have $\mathbb{E}  Z_1^{1-p} < 1$. So the above inequality \eqref{majoration eta_n} gives
(\ref{MZ eta_n}) with $C=2 \left( \mathbb{E} \left| \frac{Z_1}{m_0} -1 \right|^p \right)^{1/p} < \infty$ and $ \delta = \left( \mathbb{E}  Z_1^{1-p} \right)^{1/p} <1$.

Now we prove (\ref{exp conv log W 2}). Let $K \in (0,1)$. Using the decomposition \eqref{decomposition log Wn eta_n} and a standard truncation, we have
\begin{eqnarray} \label{A1 B1}
\mathbb{E} \left| \log W_{n+1} - \log W_n \right|
&=&
\mathbb{E} \left| \log \left( 1 +  \eta_{n} \right) \right| \mathds{1}( \eta_{n} \geq - K )
+ \mathbb{E} \left| \log \left( 1 +  \eta_{n} \right) \right| \mathds{1}(\eta_{n} < - K ) \nonumber \\
&=& A_n+ B_n.
\end{eqnarray}
We first find a bound for $A_n$. It is obvious that there exists a constant $C>0$ such that for all $x > -K$,  $| \ln (1+x) | \leq C |x|$. By (\ref{MZ eta_n}), we get
\begin{eqnarray}
\label{A_n}
A_n
\leq C \mathbb{E}|\eta_n|  
\leq C \left(\mathbb{E}|\eta_n|^p \right)^{1/p}   
\leq C \delta^{n} .
\end{eqnarray}
Now we find a bound for $B_n$.
Note that by (\ref{decomposition log Wn eta_n})
and Lemma \ref{moment log W}, we have, for any $r \in (0, q/2)$,
\begin{equation}
\label{majoration  sup E log (1+eta_n)}
\sup_{n \in \mathbb{N}} \ \mathbb{E} \left| \log (1+ \eta_n) \right|^{r} < \infty .
\end{equation}
Let $r,s>1$ be such that $ \frac{1}{s} + \frac{1}{r} =1$ and $r<q/2$. By H\"older's inequality, (\ref{majoration  sup E log (1+eta_n)}), Markov's inequality and (\ref{MZ eta_n}), we have
\begin{eqnarray}
\label{B_n}
B_n &\leq& \left( \mathbb{E} \left| \log \left( 1 + \eta_{n} \right) \right|^r \right)^{1/r} \mathbb{P} \left( \eta_{n} < -K \right)^{1/s} \nonumber \\
&\leq& C  \mathbb{P} \left( | \eta_{n}| > K \right)^{1/s} \nonumber \\
&\leq& C \left( \mathbb{E} |\eta_{n}|^p \right)^{1/s} \nonumber \\
&\leq & C \delta^{n} .
\end{eqnarray}
Thus by (\ref{A1 B1}), (\ref{A_n}) and (\ref{B_n}), there exist two constants $C>0$ and $\delta \in (0,1)$ such that
\begin{equation}
\mathbb{E} \left| \log W_{n+1} - \log W_n \right| \leq C \delta^n .
\end{equation}
Using the triangular inequality, we have for all $k \in \mathbb{N}$,
\begin{eqnarray*}
\mathbb{E} \left| \log W_{n+k} - \log W_n \right| &\leq& C \left(\delta^n + \ldots + \delta^{n+k-1} \right) \\
&\leq & \frac{C}{1- \delta} \delta^n.
\end{eqnarray*}
Letting $k \to \infty$, we get
\begin{equation*}
\mathbb{E} \left| \log W -  \log W_n \right| \leq \frac{C}{1- \delta} \delta^n,
\end{equation*}
which proves Lemma \ref{lemma logW cv exp}.
\end{proof}
We now prove a concentration inequality for the joint law of $(S_n, \log Z_n)$.
\begin{lemma}
\label{concentration lemma 1}
Assume \textbf{A1} and \textbf{A2}. Then for all $x \in \mathbb{R} $, we have
\begin{equation}  \label{concentration inequality 1}
\mathbb{P} \left( \frac{\log Z_n - n \mu}{\sigma \sqrt{n}} \leq x, \frac{S_n
- n\mu}{\sigma \sqrt{n}}\geq x \right) \leq \frac{C}{\sqrt{n}}
\end{equation}
and
\begin{equation}  \label{concentration inequality 1 bis}
\mathbb{P} \left( \frac{\log Z_n - n \mu}{\sigma \sqrt{n}} \geq x, \frac{S_n
- n\mu}{\sigma \sqrt{n}}\leq x \right) \leq \frac{C}{\sqrt{n}}.
\end{equation}
\end{lemma}

Before giving  the proof of Lemma \ref{concentration lemma 1},  let us
give some heuristics of the proof,  following Kozlov \cite{kozlov2006large}.
By (\ref{eta_n}), we can write
\begin{equation}
W_{n+1} = W_n \times \left( Z_n^{-1} \sum_{i=1}^{Z_n} \frac{N_{n,i}}{m_n} \right).
\end{equation}
Since $Z_n \rightarrow \infty$, by the law of large numbers, $ Z_n^{-1} \sum_{i=1}^{Z_n} \frac{N_{n,i}}{m_n}$ 
is close to $1$, and then $W_{n+1}/W_n$ is also close to $1$ when $n$ is large enough.
Therefore we can hope to replace $\log W_n$ by $\log W_m$ without loosing too much, when $m=m(n)$ is an increasing subsequence of integers such that $m/n \to 0$.
Denote
\begin{equation}
Y_{m,n}=\sum_{i=m+1}^{n}\frac{X_{i} -\mu }{\sigma\sqrt{n}},
\quad   Y_{n} = Y_{0,n}   \quad \text{and}\quad
V_{m} = \frac{\log W_{m}} { \sigma \sqrt{n}}.  \label{Yn Vn}
\end{equation}
Then, the
independence between $Y_{m,n}$ and $(Y_{m},
V_m)$ allows us to use a Berry-Esseen approximation on $%
Y_{m,n}$ to get the result.
\begin{proof}[Proof of Lemma \ref{concentration lemma 1}]
We first prove (\ref{concentration inequality 1}).
Let
$\alpha_{n}=\frac{1}{\sqrt{n}}$ and  $m=m(n)=\left[ n^{1/2} \right]$, where $[x]$ stands for the integer part of $x$. 
Let $D_{m}=V_{n}-V_{m}$. By a standard truncation, using Markov's inequality and Lemma \ref{lemma logW cv exp}, there exists $\delta \in (0,1)$ such that
\begin{eqnarray}
\mathbb{P}\left( Y_{n}+ V_n \leq x,Y_{n} \geq x\right)
&\leq & \mathbb{P}\left( Y_{n}+ V_m  \leq x + \alpha_n  ,Y_{n} \geq x \right)
+\mathbb{P}\left( |D_{m}|>\alpha _{n}\right) \notag \\
&\leq& \mathbb{P}\left( Y_{n} + V_m \leq x + \alpha_n ,Y_{n} \geq x \right) + \delta^{m}.
\label{ppp002}
\end{eqnarray}%
Now we find a bound for the right-hand side of (\ref{ppp002}).
Obviously we have the decomposition
\begin{equation}
Y_{n} = Y_{m} + Y_{m,n}.  \label{Yn + Ym}
\end{equation}%
For $x \in \mathbb R$, let
$G_{m,n}(x)=\mathbb{P}\left( Y_{m,n}\leq x\right) $ and $G_n(x) = G_{0,n}(x)$.  Denote by
$\nu_m (ds, dt) = \mathbb{P} \left( Y_{m}\in ds,V_{m}\in dt\right)$ the joint law of $(Y_m, V_m)$.
By conditioning and using the independence between $Y_{m,n}$ and $(Y_{m},V_{m})$, we have
\begin{eqnarray}
\label{kozlov11}
&&\mathbb{P}\left( Y_{n}+V_m \leq x + \alpha_n,Y_{n}\geq x \right)   \notag  \\
&=&\mathbb{P}\left( Y_{m,n}+Y_{m} +V_m \leq x+ \alpha_n, Y_{m,n}+Y_{m} \geq x \right)   \notag \\
&=&\int \mathbb{P}\left( Y_{m,n}+s+t \leq x+ \alpha_n,Y_{m,n}+s \geq x \right)
\nu_m (ds, dt)   \notag \\
&=&\int \mathds{1}(t \leq \alpha _{n})\left( G_{m,n}(x-s-t + \alpha_n)-G_{m,n}(x-s)\right)  \nu_m (ds, dt).
\end{eqnarray}%
For the terms $G_{m,n}(x-s-t + \alpha_n)$ and $G_{m,n}(x-s)$ we are going to use the normal approximation using the Berry-Esseen theorem.
Since $ (1-x)^{-1/2} = 1 + \frac{x}{2} + o(x) \; (x\rightarrow 0)$, we have $\frac{\sqrt{n}}{\sqrt{n-m}}= (1- \frac{m}{n})^{-1/2} =  1 + R_n$, where $0\leq  R_n  \leq C / \sqrt{n} $ and $n\geq 2$.
Therefore, we obtain
  $$ G_{m,n} (x) = \mathbb{P}\left( \sum_{i=m+1}^{n}\frac{X_{i} -\mu }{\sigma\sqrt{n }} \leq x \right) = G_{n-m} \left( \frac{x \sqrt{n} }{\sqrt{n-m}}  \right) = G_{n-m} (x (1+R_n)). $$
Furthermore, by the mean value theorem, we have
\begin{equation}
\label{mean value Phi}
\left|  \Phi (x (1+ R_n) )  - \Phi (x) \right|  \leq  R_n |x \Phi' (x)|  \leq \frac{ R_n e^{-1/2}}{ 2\pi } \leq \frac{C}{\sqrt{n}},
\end{equation}
where we have used the fact that
the function $x \mapsto x \Phi' (x) = x e^{-x^2/2}/ \sqrt{2 \pi} $ attains its maximum at $x = \pm 1$.
Therefore, by the Berry-Esseen theorem, we have for all $x\in \mathbb R$,
\begin{eqnarray}
\label{Berry-Esseen uniforme Gmn2}
\left| G_{m,n}(x) - \Phi (x)  \right|
 &\leq& \left| G_{n- m}(x (1+R_n)) - \Phi (x (1+ R_n) )  \right| + \left|  \Phi (x (1+ R_n) )  - \Phi (x) \right|\notag \\
&\leq& \frac{C}{\sqrt{n}}.
\end{eqnarray}%
From  this and \eqref{kozlov11},
we get
\begin{eqnarray}
& & \mathbb{P} \left(Y_n  + V_m \leq x + \alpha_n , Y_n \geq x  \right)  \nonumber \\
&\leq& \int \mathds{1}(t \leq \alpha _{n}) \; \left| \Phi (x-s-t + \alpha_n)-\Phi (x-s)\right| \; \nu_m (ds, dt)
  + \frac{C}{\sqrt{n}}.
\label{kozlov21}
\end{eqnarray}%
Using again the mean value theorem and the fact that  $| \Phi' (x)|  \leq 1$, we obtain
\begin{equation} \label{eq-phi-diff}
| \Phi (x-s-t + \alpha_n ) - \Phi (x-s)| = |-t + \alpha_n| \leq |t|  + \frac{1}{\sqrt n}.
\end{equation}
Moreover, by Lemma \ref{moment log W} and the definition of $\nu_m$, we have 
\begin{equation}
\label{E V_m}
\int  |t| \  \nu_m (ds, dt) = \frac{\mathbb{E} | \log W_m |}{\sigma \sqrt{n}} \leq \frac{C}{\sqrt{n}}.
\end{equation}
Hence, from  \eqref{kozlov21} and \eqref{eq-phi-diff}, we get
\begin{eqnarray*}
 \mathbb{P} \left(Y_n  + V_m \leq x+ \alpha_n, Y_n \geq x  \right)
&\leq& \frac{C}{\sqrt{n}}.
\end{eqnarray*}
Implementing this bound into (\ref{ppp002}) gives (\ref{concentration inequality 1}). The inequality
 (\ref{concentration inequality 1 bis})  is obtained in the same way.
\end{proof}

\subsection{Proof of Theorem \protect\ref{BE theorem for BPRE}}
\label{SecBE3}
In this section we prove a Berry-Esseen bound for $\log Z_n$ using Stein's method.
In order to simplify the notational burden, let
\begin{equation*}
Y_n = \frac{1}{\sigma \sqrt{n}} \sum_{i=1}^n
\; (X_i- \mu) , \ \ V_n= \frac{\log W_n}{\sigma \sqrt{n}} \quad \text{and} \quad \tilde{Y}_n=
\frac{ \log Z_n - n \mu }{\sigma \sqrt{n}} = Y_n + V_n.
\end{equation*}
By (\ref{distance kolmorogov stein}), it is enough to find a suitable bound of Stein's
expectation
\begin{equation}  \label{stein}
| \mathbb{E} [ \mathnormal{f}_x ^{\prime }(\tilde{Y}_n) - \tilde{Y}_n \mathnormal{f}_x (\tilde{Y}_n) ] |,
\end{equation}
where $x \in \mathbb{R}$ and $\mathnormal{f}_x $ is the unique bounded solution of Stein's equation (\ref{stein equation}). 
For simplicity, in the following  we  write $f$ for $f_x$.
By the triangular inequality, we have
\begin{eqnarray}  \label{steinb0}
| \mathbb{E} [ f^{\prime }(\tilde{Y}_n) - \tilde{Y}_n \mathnormal{f} (\tilde{Y}_n) ] | &\leq& 
| \mathbb{E} [ f^{\prime }(\tilde{Y}_n) - Y_n f (Y_n) ) ] | \nonumber\\
&& + | \mathbb{E} [ Y_n f (Y_n) - Y_n\mathnormal{f} (\tilde{Y}_n) ] | + | \mathbb{E} [ V_n
\mathnormal{f} (\tilde{Y}_n) ] |.
\end{eqnarray}
By \textbf{A1} and Lemma \ref{moment log W}, we have $\underset{n}{\sup}\ \mathbb{E} \left| \log
W_n \right|^{3/2} < \infty $. Therefore, by the definition of $V_n$, we have
\begin{equation}  \label{steinb1}
| \mathbb{E} [ V_n f(\tilde{Y}_n) ] | \leq \frac{\| \mathnormal{f} \|}{\sqrt{n}} \
\underset{n}{\sup} \ \mathbb{E} | \log W_n | \leq \frac{C}{\sqrt{n}}.
\end{equation}
Moreover, using the fact that $f$ is a Lipschitz function with $\|f' \| \leq 1$, together with H\"older's inequality and Lemma \ref{lemma MZ}, we get 
\begin{eqnarray}
\label{steinb3}
| \mathbb{E} [ Y_n f (Y_n) - Y_n
\mathnormal{f} (\tilde{Y}_n) ] | & \leq &
\mathbb{E} [ |Y_n| \ | \mathnormal{f} (\tilde{Y}_n) - \mathnormal{f} (Y_n)| ] \nonumber \\
& \leq & \| \mathnormal{f}^{\prime } \| \ \mathbb{E} \left[ \left| Y_n \right| \
\left| V_n \right| \right] \nonumber \\
&\leq&   \frac{1}{n}  \left[ \mathbb{%
E} \left| \sum_{i=1}^{n} (X_i - \mu) \right|^3 \right]^{1/3}  \left[ \mathbb{E}
\; \left| \log W_n \right|^{3/2} \right]^{2/3} \nonumber \\
&\leq&
 \frac{C}{n}  \left( B_3^3  \; \mathbb{E} |X_1-\mu|^3 n^{3/2} \right)^{1/3} \nonumber \\
&\leq& \frac{C }{\sqrt{n}}.
\end{eqnarray}
Again, by the triangular inequality, we have
\begin{eqnarray}  \label{steinb4}
| \mathbb{E} [ \mathnormal{f} ^{\prime }(\tilde{Y}_n) - Y_n \mathnormal{f} (Y_n) ] | 
 &\leq&  | \mathbb{E} [ \mathnormal{f}^{\prime }(Y_n+V_n) - \mathnormal{f}%
^{\prime }(Y_n) ] |  \nonumber\\
&& + | \mathbb{E} [ \mathnormal{f}^{\prime }(Y_n) - Y_n
\mathnormal{f} (Y_n) ] |.
\end{eqnarray}
Applying (\ref{majoration fx'}) for $w=Y_n$, $s=V_n$ and $t=0$, we get
\begin{eqnarray*}
| \mathbb{E} [ \mathnormal{f}^{\prime }(Y_n+V_n) - \mathnormal{f} ^{\prime }(Y_n)
] | &\leq& \mathbb{E} \left( | Y_n | |V_n| \right) +  \mathbb{E} | V_n | + \mathbb{P}
\left( Y_n + V_n \leq x, Y_n \geq x \right) \\
&& + \mathbb{P} \left( Y_n + V_n \geq x, Y_n
\leq x \right).
\end{eqnarray*}
As for (\ref{steinb1}) and (\ref{steinb3}), we have $ \mathbb{E} | V_n | \leq \frac{C}{\sqrt{n}}$ and $\mathbb{E} \left( | Y_n | \; |V_n| \right) \leq \frac{C}{\sqrt{n}}.$
From these bounds and the concentration inequalities of Lemma \ref{concentration lemma 1}, we have
\begin{equation}  \label{steinb5}
| \mathbb{E} [ \mathnormal{f}^{\prime }(Y_n+V_n) - \mathnormal{f} ^{\prime }(Y_n) ]
| \leq \frac{C}{\sqrt{n}}.
\end{equation}
Furthermore, since $Y_n$ is a sum of i.i.d.\ random variables, by Lemma \ref{lem borne eq stein}, it follows that
\begin{equation}
\label{steinb6}
| \mathbb{E} [ \mathnormal{f}^{\prime }(Y_n) - Y \mathnormal{f} (Y_n) ] | \leq \frac{C}{\sqrt{n}}.
\end{equation}
Thus, coming back to (\ref{steinb0}) and using the bounds (\ref{steinb1}), (\ref{steinb3}), (\ref{steinb4}), \eqref{steinb5} and \eqref{steinb6}, we get
\begin{equation*}
| \mathbb{E} [ \mathnormal{f} ^{\prime }(\tilde{Y}_n) - \tilde{Y}_n \mathnormal{f} (\tilde{Y}_n) ] | \leq \frac{C}{\sqrt{n}},
\end{equation*}
which ends the proof of Theorem \ref{BE theorem for BPRE}.

\section{Harmonic moments of $W$}
\label{Sec Harm Mom}
In this section, we study the existence of harmonic moments of the random variable $W$. Section \ref{sec harmonic moment 1} is devoted to the proof of Theorem \ref{harmonic moment}.
For the needs of Cram\'er's type large deviations, in Section \ref{sec harmonic moment 2} 
 we  shall prove the existence of the harmonic moments of $W$ under the changed probability measure, 
 which  generalizes the result of Theorem  \ref{harmonic moment}.

\subsection{Existence of harmonic moments under  $\mathbb{P}$}
\label{sec harmonic moment 1}
Following the line of Lemma \ref{moment log W} we prove Theorem \ref{harmonic moment} by studying the asymptotic behavior of the Laplace transform of $W$.
Actually Theorem \ref{harmonic moment} is a simple consequence of Theorem \ref{Laplace transform of W} below.
Recall that
\begin{equation*}
\phi_{\xi} (t) = \mathbb{E}_{\xi} e^{-t W} \ \ \ \text{and} \ \ \ \phi (t) =
\mathbb{E} \phi_{\xi} (t) = \mathbb{E} e^{-t W}, \quad t \geq 0.
\end{equation*}

\begin{theorem}
\label{Laplace transform of W} Assume condition \textbf{A3}.
Let $a_0>0$ be defined by \eqref{Def a0}.
Then for any $a \in (0,a_0) $,  there exists a constant $C>0$  such that for all $t > 0$,
\begin{equation*}
\phi (t) \leq C t^{-a}.
\end{equation*}
In particular $\mathbb{E} W^{- a } < \infty $ for all $a \in (0,a_0)$.

\end{theorem}
\begin{proof}
By (\ref{alpha0}), we have for $A>1$ and $ t \geq K A^n$,
\begin{equation}
\phi (t) \leq \alpha^n + \mathbb{P}( \Pi_n \geq A^n ) ,
\label{alpha01}
\end{equation}
where
\begin{equation}
\alpha = \mathbb{E} p_1 (\xi) + (1-  \mathbb{E} p_1 ( \xi) ) \gamma  \in (0,1).
\label{alpha1}
\end{equation}
Using Markov's inequality and condition \textbf{A3}, there exists $\lambda_0 > 0 $ such that
\[ \mathbb{P} ( \Pi_n \geq A^n) \leq \frac{\mathbb{E} \Pi_n^{\lambda_0}}{A^{n \lambda_0}} = \left( \frac{\mathbb{E} m_0^{\lambda_0}}{A^{\lambda_0}} \right)^n . \]
Setting $ A = \left(\frac{\mathbb{E} m_0^{\lambda_0}}{\alpha} \right)^{1/ \lambda_0} >1, $ we get for any $ n \in \mathbb{N} $ and $ t \geq K A^n$,
\begin{equation}
 \phi (t) \leq 2 \alpha^n.
\label{bbb001}
\end{equation}
Now, for any  $t \geq K$, define $n_0 = n_0 (t) = \left[ \frac{\log (t/K)}{\log A } \right] \geq 0$, where $ [x]$ stands for the integer part of $x$, so that
\[  \frac{\log (t/K)}{\log A} - 1 \leq n_0 \leq \frac{\log (t/K)}{\log A} \ \  \text{and} \ \  t \geq K A^{n_0}  .\]
Then, for $t \geq K$,
\[\phi(t) \leq 2 \alpha^{n_0} \leq 2 \alpha^{-1} (t/K)^{\frac{\log \alpha }{\log A }} = C_0 t^{-a}, \]
with $C_0 =  2 \alpha^{-1} K^{a}$ and $ a = - \frac{\log \alpha }{\log A } > 0 $.
Thus we can choose a constant $C >0$ large enough such that for all $t > 0$,
\begin{equation}
\label{majoration phi}
\phi (t) \leq C t^{-a}.
\end{equation}
This proves the first inequality of Theorem \ref{Laplace transform of W}.
The existence of harmonic moments of $W$ of order $s \in (0,a)$ is deduced from \eqref{majoration phi} and the fact that 
\[
\e W^{-s} = \frac{1}{\Gamma (s)} \int_0^{+ \infty}  \phi (t)  t^{s-1} dt,
\]
where $\Gamma$ is the Gamma function. 

Now we prove (ii). By the definition of $a$, $A$ and $\alpha$, we have
\begin{eqnarray*}
a &=& - \lambda_0 \frac{ \log \alpha}{\log \mathbb{E} m_0^{\lambda_0} -
\log \alpha} \\
&=& - \lambda_0 \frac{ \log \left( \mathbb{E} p_1 + (1- \mathbb{E} p_1 )
\gamma \right) }{\log \mathbb{E} m_0^{\lambda_0} - \log \left( \mathbb{E}
p_1 + (1- \mathbb{E} p_1 ) \gamma \right) },
\end{eqnarray*}
where $\gamma \in (0,1)$ is an arbitrary constant.
Since $a \to a_0$ as $\gamma \to 0$, this concludes the proof of Theorem \ref{Laplace transform of W}.
\end{proof}

\subsection{Existence of  harmonic moments under $\mathbb{P}_{\lambda}$\label{sec harmonic moment 2}}

In this section, we establish a uniform bound for the harmonic moments of $W$ under the probability measures
$\mathbb{P}_{\lambda}$, uniformly in $\lambda \in [0,\lambda_0] $.

Let $m(x) = \mathbb{E}[Z_1 | \xi_0 = x] = \sum_{k=1}^{\infty} k p_k (x) $.
By \textbf{A3}, for all $\lambda \leq  \lambda_0$, we can define the conjugate distribution function $\tau_{0, \lambda}$ as
\begin{equation}  \label{tau tilde}
\tau_{0,\lambda} (dx) = \frac{m(x)^\lambda}{ L(\lambda) } {\ \tau_0 (dx)}.
\end{equation}
Note that (\ref{tau tilde}) is just Cram\'er's change of measure for the associated random walk $(X_n)_{n \geq 1
}$.
Consider the new branching process in a random environment whose environment
distribution is $\tau_{\lambda} = \tau_{0, \lambda}^{\otimes \mathbb{N}}$.
The corresponding annealed probability and expectation are denoted by
\begin{equation}
\mathbb{P}_{\lambda}(dx,d\xi) = \mathbb{P}_{\xi}(dx) {\tau}_{\lambda} (d\xi)
\label{Def Plambda}
\end{equation}
and
$\mathbb{E}_{\lambda} $ respectively.
Note that, for any $\mathcal{F}_n$-measurable random variable $T$, we have
\begin{equation}
\mathbb{E}_{\lambda} T  = \frac{\mathbb{E} e^{ \lambda S_n} T }{L \left(\lambda \right)^n}.
\label{Def chmt measure}
\end{equation}
It is easily seen that under $ \mathbb{P}_{\lambda}$,   the process $(Z_n)$  is still a supercritical branching process in a random environment, 
which verifies the condition (\ref{p0}), and that
 $(W_n)_{n \in \mathbb{N}}$ is still a non-negative martingale which converges a.s.\ to $W$.
We shall show under the additional assumption \textbf{A4} that there exists a constant $a >0$ such that for all $b \in (0,a)$,
\[ \sup_{ 0 \leq  \lambda \leq \lambda_0 } \mathbb{E}_{\lambda} W^{-b} < \infty . \]
Denote the Laplace transforms of $W$ under  $\mathbb{P}_{\lambda}$  by
\begin{equation*}
\phi_{\lambda} (t) = \mathbb{E}_{\lambda } \phi_{\xi} (t) = \mathbb{E}_{\lambda} e^{-t W},
\end{equation*}
where $t \geq 0$ and $ \lambda  \leq \lambda_0$. The following  theorem gives a  bound on $ \phi_{\lambda} (t)$
and $ \mathbb{E}_{\lambda} W^{- a }$ uniformly in $\lambda \in [0, \lambda_0]$.

\begin{theorem} \label{moment harmonique W uniforme} Assume conditions \textbf{A3} and \textbf{A4}.
Then there exist constants $a>0$ and $C>0$ such that  for all $t > 0$,
\begin{equation*}
\sup_{0 \leq \lambda \leq \lambda_0} \phi_{\lambda} (t) \leq C
t^{-a}.
\end{equation*}
In particular, we have $ \underset{0 \leq \lambda \leq \lambda_0}{\sup}  \mathbb{E}_{\lambda} W^{- b } < \infty $ for all $b \in (0,a)$.
\end{theorem}
For the proof of the previous theorem we need to control the exponential speed of convergence in $L^p$ of $W_n$ to $W$, uniformly under the class of probability measures $\l \bb{P}_{\lambda} \r _{0 \leq \lambda \leq \lambda_0}$.
\begin{lemma} \label{convergence Lp exponentielle uniforme}
Assume that \textbf{A3} holds for some $\lambda_0 >0 $, and
\textbf{A4} holds for some $p \in (1,2]$. Then for $\lambda_0 >0 $ small enough, there exist constants $C>0$ and $\delta_0 \in (0,1) $ such that, for all $n \geq 1$,
\begin{eqnarray*}
\sup_{0 \leq \lambda  \leq \lambda_0}  \left(\mathbb{E}_{\lambda} \left| W_n-W \right|^p \right)^{1/p}   \leq C \delta_0^n .
\end{eqnarray*}
\end{lemma}
\begin{proof}
Applying Lemma \ref{lemma MZ} under $\mathbb{E}_{\xi} $ to the decomposition (\ref{difference Wn+1-Wn})
and using the independence between $Z_n$ and $ \frac{N_{n,i}}{m_n} \left( i \geq 1 \right)$, we get
\begin{eqnarray*}
\mathbb{E}_{\xi} \left| W_{n+1} - W_n \right|^p &\leq& 2^p \Pi_n^{-p} \mathbb{E}_{\xi} Z_n \mathbb{E}_{\xi} \left| \frac{N_n}{m_n} -1 \right|^p \\
&=& 2^p \; \Pi_n^{1-p}  \; \mathbb{E}_{\xi} \left| \frac{N_n}{m_n} -1 \right|^p .
\end{eqnarray*}
Note that under $\mathbb{P}_{\lambda}$, the r.v.'s $m_0, \ldots, m_{n-1}$ are i.i.d., independent of 
$\frac{N_n}{m_n},$ and $\frac{N_n}{m_n}$ has the same law as $\frac{Z_1}{m_0}$.
Thus, taking expectation $\mathbb{E}_{\lambda}$, we get
\begin{eqnarray} \label{majoration accroissement Wn 1}
\mathbb{E}_{\lambda} \left| W_{n+1} - W_n \right|^p &\leq& 2^p \left( \mathbb{E}_{\lambda} m_0^{1-p} \right)^n \mathbb{E}_{\lambda} \left| \frac{Z_1}{m_0} -1 \right|^p .
\end{eqnarray}
Recall that $m_0>1$. Choose $\lambda_0>0$ small enough such that
$ p - \lambda_0 > 1$.
By condition  \textbf{A4}, for all $0 \leq \lambda \leq \lambda_0$,
\begin{equation*}
\mathbb{E}_{\lambda} \left( \frac{Z_1}{m_0} \right)^p
= \frac{1}{\mathbb{E}  m_0^{\lambda} } \mathbb{E} \left( \frac{Z_1^p} {m_0^{p-\lambda} } \right)
\leq \mathbb{E} \left( \frac{Z_1^p} {m_0}   \right)  < + \infty.
\end{equation*}
Since $ 1-p+ \lambda_0 <0$, we have, for all $0 \leq \lambda \leq \lambda_0$, $ \mathbb{E}_{\lambda} m_0^{1-p} = \frac{1}{\mathbb{E}  m_0^{\lambda} } \mathbb{E} m_0^{1-p+\lambda}  \leq
         \mathbb{E} m_0^{1-p+\lambda_0} <1$. Hence
by (\ref{majoration accroissement Wn 1}),     for  $\delta_0=\left( \mathbb{E} m_0^{1-p+ \lambda_0} \right)^{1/p} <1 $ and
$ C = 2 \left( \mathbb{E} \left(\frac{Z_1^p}{m_0}\right)^{1/p}  +1 \right)  <\infty$, we have
\begin{eqnarray} \label{majoration accroissement Wn 2}
\sup_{  0 \leq \lambda \leq \lambda_0 } \left( \mathbb{E}_{\lambda} \left| W_{n+1} - W_n \right|^p \right)^{1/p} &\leq&  C \delta_0^n.
\end{eqnarray}
Using the triangular inequality, for all $k \in \mathbb{N}$,
\begin{eqnarray*}
\sup_{ 0 \leq \lambda  \leq \lambda_0 } \left( \mathbb{E}_{\lambda} \left| W_{n+k} - W_n \right|^p \right)^{1/p} &\leq&  C \left( \delta_0^n + \ldots + \delta_0^{n+k-1} \right) \\
&\leq & \frac{ C}{1- \delta_0} \delta_0^n.
\end{eqnarray*}
Letting $k \to \infty$, we get
\begin{eqnarray} \label{majoration Lp uniforme W-Wn}
\sup_{  0 \leq \lambda  \leq \lambda_0 } \left( \mathbb{E}_{\lambda} \left| W - W_n \right|^p \right)^{1/p} &\leq & \frac{C}{1- \delta_0} \delta_0^n,
\end{eqnarray}
which concludes the proof of  Lemma \ref{convergence Lp exponentielle uniforme}.
\end{proof}
Now we proceed to prove Theorem \ref{moment harmonique W uniforme}.
\begin{proof}[Proof of Theorem \ref{moment harmonique W uniforme}]
Let $\varepsilon \in (0, 1)$.
By a truncation argument, we have for all $\lambda \in [0, \lambda_0]$, and $n \in \mathbb{N}$,
\begin{eqnarray}
\label{troncature W Wn}
\phi_{\lambda} (t)
&=& \mathbb{E}_{\lambda} e^{-tW}\left[ \mathds{1} \left( | W_n - W | \leq  \varepsilon^n \right) + \mathds{1} \left( | W_n - W | > \varepsilon^n \right) \right] \nonumber \\
&\leq& e^{t \varepsilon^n} \mathbb{E}_{\lambda} e^{-t W_n}  + \mathbb{P}_{\lambda} ( |W_n-W| > \varepsilon^n ).
\end{eqnarray}
Using Markov's inequality and Lemma \ref{convergence Lp exponentielle uniforme}, there exists $\delta_0 \in (0,1)$ such that
\begin{equation}
\label{beta1}
\sup_{  0 \leq \lambda \leq \lambda_0 } \mathbb{P}_{\lambda} ( |W_n-W| > \varepsilon^n )\leq  C {\beta_1^n},
\end{equation}
where $\beta_1= \delta_0 / \varepsilon <1$ for $\varepsilon>\delta_0.$

Now we proceed to bound the first term in the right-hand side of \eqref{troncature W Wn}. 
Recall that $L(\cdot)$ is increasing. 
Furthermore, since $x \mapsto e^{-t x}$ is a non-negative and convex function, we have (see Lemma 2.1 of \cite{liu}) that $\sup_{n \in \mathbb{N}}  \mathbb{E} e^{-t W_n} = \mathbb{E} e^{-t W} = \phi (t)$.
Then, again using truncation, we have for all $ \lambda \in  [0, \lambda_0] $, $ n \in \mathbb{N}$ and $c> \mu$,
\begin{eqnarray}
\mathbb{E}_{\lambda} e^{-t W_n}
&=& \mathbb{E}_{\lambda} e^{-t W_n}\left[ \mathds{1} \left( S_n \leq  cn  \right) + \mathds{1} \left(  S_n >  cn \right) \right] \nonumber \\
&\leq& e^{\lambda_0 c n } \phi (t) + \mathbb{P}_{\lambda} (  S_n >  cn ).
\label{Elambda001}
\end{eqnarray}
By the exponential Markov's inequality, we have for $ \lambda \leq \lambda_0 / 2 $,
\begin{eqnarray*}
\mathbb{P}_{\lambda} (  S_n >  cn ) &\leq & \left( \mathbb{E}_{\lambda} e^{ \lambda X} \right)^n e^{-\lambda cn}  \\
&=& e^{n (\psi (2\lambda)- \psi (\lambda)- \lambda c)} ,
\end{eqnarray*}
where $\psi (\lambda)= \log \mathbb{E} e^{\lambda X}$ 
and $\psi (2\lambda)- \psi (\lambda)- \lambda c= \lambda \mu - \lambda c + o (\lambda)$ as $\lambda \to 0.$
Since $c> \mu$ we can choose $\lambda_0 > 0$ small enough, such that for all $0 \leq \lambda \leq \lambda_0,$
 $\psi (2\lambda)- \psi (\lambda)- \lambda \leq   \lambda (\mu - c) / 2 <
 0$. Thus we have
\begin{equation}
\sup_{0 \leq \lambda  \leq \lambda_0}\mathbb{P}_{\lambda} (  S_n >  cn ) \leq \beta_2^n,
\label{beta2}
\end{equation}
where $\beta_2 = e^{ \lambda (\mu - c) / 2 }< 1$.
Furthermore by Theorem \ref{Laplace transform of W}, for all $a
\in (0, a_0)$, there exists $C>0$ such that $ \phi (t) \leq C t^{-a}$ for all $t>0$.
Thus implementing (\ref{beta1}), (\ref{Elambda001}) and (\ref{beta2}) into (\ref{troncature W Wn}) leads to
\begin{equation}
\sup_{0 \leq \lambda   \leq \lambda_0} \phi_\lambda (t) \leq e^{t \varepsilon^n} \left(  e^{\lambda_0 c n}  C t^{-a}    +   \beta_2^n  \right)   + C\beta_1^n.
\label{aaa001}
\end{equation}
Since $\phi_\lambda(t)$ is decreasing in $t$, we have for any $t\geq t_n=\varepsilon^{-n},$
\begin{equation}
\sup_{0 \leq \lambda  \leq \lambda_0} \phi_\lambda (t) \leq  \sup_{0 \leq \lambda   \leq \lambda_0} \phi_\lambda (t_n) \leq e \left(  e^{\lambda_0 c n}  C \varepsilon^{an}    +   \beta_2^n  \right)   + C\beta_1^n.
\label{aaa002}
\end{equation}
Choosing $\lambda_0 > 0$ small enough such that $\beta_3 = e^{\lambda_0 c} \varepsilon^a <1$, we find that there exists a constant $C>0$ and $\beta= \max \left\{ \beta_1, \beta_2, \beta_3 \right\} \in (0,1)$ such that, for any $t\geq \varepsilon^{-n},$
\begin{equation}
\sup_{0 \leq \lambda   \leq \lambda_0} \phi_\lambda (t)  \leq  C\beta^n.
\label{aaa003}
\end{equation}
The rest of the proof is similar to that of Theorem \ref{harmonic moment}, starting from (\ref{bbb001}).
\end{proof}

\section{Proof of Cram\'er's large deviation expansion}
In this section, we prove Theorem \ref{cramer type theorem}.
The starting point is the decomposition \eqref{decom-logZn}.
We will show that the Cram\'er-type large deviation expansion of $\log Z_n$ is determined by that of the associated random walk $(S_n)$.
Our proof is based on Cram\'er's change of measure $\mathbb{P}_{\lambda}$ defined by (\ref{Def Plambda}). 
 An important step in the approach is to have a good  control  of the joint law of the couple $(S_n ,\log Z_n)$ under the changed measure 
 $ \mathbb{P}_{\lambda}$ uniformly in $ \lambda \in [0,\lambda_0]$, for some small $\lambda_0,$ 
which  is done in Section \ref{section lemma}. 
The proof of Theorem \ref{cramer type theorem} is deferred to Section \ref{section proof}.

In the sequel we shall use the  first three moments of the r.v.\ $X= \log m_0 $ 
under the changed probability measure $\mathbb{P}_{\lambda }$:
\begin{eqnarray}
\mu _{\lambda } &=& \mathbb{E}_{\lambda }X = \psi' (\lambda )    =\sum_{k=1}^{\infty }\frac{\gamma _{k}%
}{(k-1)!}\lambda^{k-1}, \label{mu serie}\\
\sigma _{\lambda } &=& \mathbb{E}_{\lambda }\left(X -\mu _{\lambda}\right) ^{2} = \psi'' (\lambda )  =\sum_{k=2}^{\infty }\frac{\gamma _{k}%
}{(k-2)!}\lambda^{k-2},  \label{sigma serie} \\
\rho_{\lambda } &=&  \mathbb{E}_{\lambda } | X -\mu_{\lambda }|^3,
\end{eqnarray}
with $ \psi $ defined in \eqref{def-phi}.

\subsection{Auxiliary results} \label{section lemma}
In this section we prove a uniform concentration inequality bound for the class of probability measures $\left( \mathbb{P}_{\lambda} \right)_{0 \leq \lambda \leq \lambda_0} $.
First we give uniform bounds for the first three moments of $X$ under $\mathbb{P}_{\lambda}$.
It is well known that, for $\lambda_0$ small enough
and for any $\lambda \in \left[ 0, \lambda_0 \right],$
\begin{equation}
\label{3 moments uniformes}
\left| \mu_ \lambda  - \mu \right| \leq C_1 \lambda , \qquad \left| \sigma_ \lambda  - \sigma \right| \leq C_2 \lambda ,
\qquad \left|  \rho_ \lambda  - \rho \right| \leq C_3 \lambda,
\end{equation}
where $C_1, C_2, C_3$ are absolute constants.
These bounds allow us to obtain an uniform rate of convergence for the process $(\log W_n)$ under $\mathbb{P}_{\lambda}$.
\begin{lemma}
\label{lemma Wexp bis}
Assume \textbf{A3} and \textbf{A4}. Then there exists $\delta_0 \in (0,1)$ such that
\begin{equation}
\sup_{0 \leq \lambda \leq  \lambda_0}  \mathbb{E}_\lambda \left| \log W_n -\log W \right| \leq \delta_0^n.
\label{exp conv log W}
\end{equation}
\end{lemma}
\begin{proof}
The proof is similar to that in Lemma \ref{lemma logW cv exp}:
it is enough to replace  $\mathbb E$ by  $\mathbb E_\lambda$ and to ensure that all the bounds
in that proof still hold uniformly in $ \lambda \in [0 , \lambda_0] $, for $\lambda_0>0$ small enough.

We first prove that for some constants $\lambda_0>0$, $\delta \in (0,1)$ and  $C>0$,
\begin{equation}
\label{majoration eta_n 2}
\sup_{0 \leq \lambda \leq \lambda_0} \left( \mathbb{E}_{\lambda} |\eta_n|^p \right)^{1/p}  \leq C \delta^{n},
\end{equation}
where $\eta_n$ is defined \eqref{eta_n}. 
In fact,
we have, for $p \in (1,2)$,
\begin{equation*}
\mathbb{E}_{\lambda} |\eta_n|^p  \leq 2^p \mathbb{E}_{\lambda} \left| \frac{Z_1}{m_0} -1 \right|^p  \left(\mathbb{E}_{\lambda} \left[Z_1^{1-p} \right] \right)^n .
\end{equation*}
 By the dominated convergence theorem and the fact that $m_0>1$, we have $\mathbb{E}_{\lambda}  Z_1^{1-p}  \leq  \mathbb{E}  Z_1^{1-p} m_0^{\lambda} \underset{\lambda \to 0}{\longrightarrow} \mathbb{E} Z_1^{1-p} < 1.$
 Thus there exists a $\lambda_0>0$ small enough such that
\[  \mathbb{E}_{\lambda}  Z_1^{1-p} \leq \mathbb{E} Z_1^{1-p} m_0^{\lambda_0} < 1 .\]
By  \textbf{A3} and \textbf{A4}, for some small enough $\lambda_0 \in (0,p-1]$ and all $ \lambda \in [0, \lambda_0]$ we have,
\[
\mathbb{E}_{\lambda} \left( \frac{Z_1}{m_0} \right)^p = \left( \mathbb{E}m_0^{\lambda} \right)^{-1}\mathbb{E} \frac{Z_1^p}{m_0^{p- \lambda_0}} \leq \mathbb{E} \frac{Z_1^p}{m_0^{p- \lambda_0}} \leq \mathbb{E} \frac{Z_1^p}{m_0} < \infty.
\]
Therefore, (\ref{majoration eta_n 2}) holds with
$C \leq  2 \left[ \left( \mathbb{E} \frac{Z_1^p}{m_0}\right)^{1/p} +1 \right] < \infty$ and $ \delta \leq \left( \mathbb{E} Z_1^{1-p} m_0^{\lambda_0} \right)^{1/p} <1$.

Next we show that
\begin{equation}
\label{majoration  sup E log (1+eta_n) 2}
\sup_{n \in \mathbb{N}} \ \sup_{0 \leq \lambda \leq \lambda_0 } \ \mathbb{E}_{\lambda} \left| \log (1+ \eta_n) \right|^{r} < \infty,
\end{equation}
for all $r>0$.
It is easily seen that there exists a constant $C_r>0$ such that
$ \mathbb{E}_{\lambda} \left| \log W \right|^r \leq C_r \left( \mathbb{E}_{\lambda} W^{- \alpha} + \mathbb{E}_{\lambda} W \right) \leq C_r \left( \mathbb{E}_{\lambda} W^{- \alpha} + 1 \right).$
Then, by \textbf{A3} and Theorem \ref{moment harmonique W uniforme}, for all $r>0$, we have
\begin{equation}
\label{E log W}
\sup_{0 \leq \lambda \leq \lambda_0} \mathbb{E}_{\lambda} | \log W |^r < \infty. 
\end{equation}
Thus by (\ref{majoration uniforme moment log W}) and (\ref{decomposition log Wn eta_n}) we get (\ref{majoration  sup E log (1+eta_n) 2}).

We finally end the proof in the same way as in Lemma \ref{lemma logW cv exp},  using the uniform bounds (\ref{majoration  sup E log (1+eta_n) 2}) and (\ref{majoration eta_n 2}).
\end{proof}
Now we give a control of the joint law of $(S_n, \log Z_n)$ for the convergence to the distribution function $\Phi([0,x]) \mathds 1 (x \geq 0 ),$ $x \in \mathbb R, $ uniformly in $\lambda \in [0, \lambda_0]$, 
where $\Phi([0,x])=\Phi(x)-\Phi(0)$ (recall that $\Phi$ is the distribution function of the standard normal law).
\begin{lemma}
\label{ecart iid} Assume \textbf{A3} and \textbf{A4}. There exist positive constants
$C$, $\beta_1$, $\beta_2$ and $\delta \in (0,1)$ such that for any $x>0$,
\begin{equation}
\label{ecart001}
\sup_{0 \leq  \lambda  \leq \lambda_0}
\left| \mathbb{P_\lambda} \left( {\frac{S_n-n \mu }{\sigma_\lambda \sqrt n }}
\leq x, {\frac{\log Z_n-n \mu }{\sigma_\lambda \sqrt n }} \geq 0 \right) - \Phi([0,x])
\right| \leq \frac{C}{\sqrt{n}} ,
\end{equation}
and
\begin{eqnarray}
\label{ecart002}
& & \sup_{0 \leq  \lambda  \leq \lambda_0}   \mathbb{P_\lambda} \left( {\frac{S_n-n \mu }{\sigma_\lambda \sqrt n }}
\leq -x, {\frac{\log Z_n-n \mu }{\sigma_\lambda \sqrt n }} \geq 0 \right)  \nonumber \\
& & \qquad \qquad\qquad \leq C
\left( x + \frac{1}{\sqrt{n}} \right) e^{ - \beta_1 x \sqrt{n} } + \min \left( e^{ - \beta_2 x \sqrt{n} },  \delta^{\sqrt{n}} x^{-1/2} n^{-1/4} \right).
\end{eqnarray}
\end{lemma}
\begin{proof}
Let $m=m(n)=\left[ n^{1/2} \right]$, with $[x]$ denoting the integer part of $x$, and
\begin{equation*}
\label{Yn Vn bis}
Y_{m,n}^{\lambda}=\sum_{i=m+1}^{n}\frac{X_{i} -\mu_\lambda }{\sigma_\lambda \sqrt{n}},
\quad   Y_{n}^{\lambda} = Y_{0,n}^{\lambda}   \quad \text{and}\quad
V_{m}^{\lambda} = \frac{\log W_{m}} { \sigma_\lambda \sqrt{n}}.
\end{equation*}
The proof of \eqref{ecart001} is similar to that of Lemma \ref{concentration lemma 1}
with $\p$ replaced by $\p_\lambda$.
The only difference is that the bounds \eqref{Berry-Esseen uniforme Gmn2} and \eqref{E V_m} have to be uniform 
in $\lambda \in [0, \lambda_0]$. The uniformity in \eqref{Berry-Esseen uniforme Gmn2} is ensured by  the Berry-Esseen theorem 
and \eqref{3 moments uniformes} which imply that
\begin{equation}
\label{eq BE uniforme}
\sup_{\lambda \in [0, \lambda_0]} \left| G_{m,n}^{\lambda} (x) - \Phi (x) \right| \leq \frac{C}{\sqrt{n}},
\end{equation}
where $G_{m,n}^{\lambda} (x) = \p_{\lambda} \left( Y_{m,n}^{\lambda}  \leq x\right)$. The uniformity in \eqref{E V_m}  is a consequence of Lemma \ref{lemma Wexp bis}.
Further details of the proof are left to the reader.

Now we prove (\ref{ecart002}). Let $D_m^{\lambda} = V_n^{\lambda} - V_m^{\lambda}$.
By considering the events  $ \{ |D_m^{\lambda}| \leq \frac{x}{2} \}$ and $ \{ |D_m^{\lambda}| > \frac{x}{2} \}$  we have
\begin{eqnarray}
\label{troncature lemme 4.2 ii}
\mathbb{P}_{\lambda} \left( Y_n^{\lambda}  \leq -x, Y_n^{\lambda}  + V_n^{\lambda}  \geq 0 \right)
&\leq& \mathbb{P}_{\lambda} \left( Y_n^{\lambda} \leq -x, Y_n^{\lambda}  + V_m^{\lambda}  \geq -\frac{x}{2} \right) \nonumber \\
&&+ \mathbb{P}_{\lambda} \left( |D_m^{\lambda}| > \frac{x}{2} \right) .
\end{eqnarray}

We first find a suitable bound of the first term of the right-hand side  of (\ref{troncature lemme 4.2 ii}). Again by decomposing $Y_n^{\lambda} = Y_{m,n}^{\lambda} + Y_m^{\lambda} $, using \eqref{eq BE uniforme} and the fact that $\Phi ([a,b]) \leq b-a$, we have
\begin{eqnarray*}
\label{kozlov3}
&&\mathbb{P}_{\lambda} \left(Y_n^{\lambda} \leq -x, Y_n^{\lambda} + V_{m}^{\lambda} \geq -\frac{x}{2} \right) \notag \\
&=& \int \mathds{1} \left( t > \frac{x}{2} \right) \mathbb{P}_{\lambda} \left( Y_{m,n}^{\lambda} \in \left[ -\frac{x}{2} -s-t, -x-s \right]\right) \nu_m^{\lambda} (ds, dt) \notag \\
&\leq& \int \mathds{1} \left( t > \frac{x}{2} \right) \left[ \Phi \left( \left[ -\frac{x}{2} -s-t, -x-s \right]\right) +  \frac{C}{\sqrt{n}} \right] \nu_m^{\lambda} (ds, dt)   \notag \\
&\leq&  \int \mathds{1} \left( t >\frac{x}{2} \right) \left[ \left(t-\frac{x}{2}\right)+  \frac{C}{\sqrt{n}} \right] \nu_m^{\lambda} (ds, dt)   \notag \\
&\leq&   \mathbb{E}_{\lambda} \left[ V_{m}^{\lambda} \mathds{1} \left( V_{m}^{\lambda} \geq \frac{x}{2}
\right) \right] + \left[ \frac{x}{2} +  \frac{C}{\sqrt{n}} \right] \mathbb{P}_{\lambda} \left( V_{m}^{\lambda} > \frac{x}{2} \right) .
\end{eqnarray*}
By Markov's inequality, we have
$
\mathbb{P}_{\lambda} \left( V_m^{\lambda} > \frac{x}{2} \right) \leq e^{-\frac{x}{2} \sigma_{\lambda} \sqrt{n}}.
$
Moreover, using H\"older's and Markov's inequalities, we get by \eqref{E log W} and the definition of $V_m$ that
$$
\mathbb{E}_{\lambda} \left[ V_m^{\lambda} \mathds{1}\left( V_m^{\lambda} \geq \frac{x}{2} \right) \right] \leq
\left(\mathbb{E}_{\lambda} |V_m^{\lambda} |^2 \right)^{1/2} \mathbb{P}_{\lambda} \left( V_m^{\lambda}
\geq \frac{x}{2} \right)^{1/2} \leq \frac{C}{\sigma_{\lambda} \sqrt{n}} e^{-\frac{x}{4}\sigma_{\lambda} \sqrt{n}}.
$$
Since, by (\ref{3 moments uniformes}), $\sigma_\lambda$ is bounded uniformly in $\lambda \in [0,\lambda_0]$, there exists $\beta_1 >0$ such that for any $\lambda \in [0,\lambda_0]$,
\begin{equation}  \label{majoration 2}
\mathbb{P}_{\lambda} \left( Y_n^{\lambda} \leq -x, Y_n^{\lambda} + V_m^{\lambda} \geq -\frac{x}{2} \right)
\leq C \left( x + \frac{1}{\sqrt{n}} \right) e^{ - \beta_1 x  \sqrt{n}
}.
\end{equation}

We now search for a suitable bound for the second term of the right-hand side of (\ref{troncature lemme 4.2 ii}). By  H\"older's inequality and Theorem \ref{moment harmonique W uniforme}, there exist some constants $C>0$, $a>0$ and $0 < \alpha < \min (1/2, a/2)$ such that, for all $\lambda \in [0, \lambda_0]$ and $n \in \mathbb{N}$,
$$
\mathbb{E}_{\lambda} \left(\frac{W_n}{W_m}\right)^{\alpha} \leq \left(\mathbb{E}_{\lambda} W_n^{2\alpha} \right)^{1/2} \left(\mathbb{E}_{\lambda} W_m^{-2\alpha} \right)^{1/2} \leq  \left(\mathbb{E}_{\lambda} W^{2\alpha} \right)^{1/2} \left(\mathbb{E}_{\lambda} W^{-2\alpha} \right)^{1/2} \leq C.
$$
Thus, by Markov's inequality and (\ref{3 moments uniformes}), there exists a constant $\beta_2 >0$ (independent of $(\lambda, n, x)$) such that, for all $\lambda \in [0, \lambda_0]$,
\begin{eqnarray}
\label{majoration 3 bis}
\mathbb{P}_{\lambda} \left( | D_m^{\lambda} | > \frac{x}{2} \right) &\leq& \mathbb{P}_{\lambda} \left( \left(\frac{W_n}{W_m}\right)^{\alpha} > e^{\alpha \sigma_{\lambda} \sqrt{n} \frac{x}{2} } \right) + \mathbb{P}_{\lambda} \left( \left(\frac{W_m}{W_n}\right)^{\alpha} > e^{\alpha \sigma_{\lambda} \sqrt{n} \frac{x}{2} } \right) \nonumber\\
&\leq& C e^{ - \beta_2 x \sqrt{n} }.
\end{eqnarray}
Moreover, by Markov and Jensen's inequalities and Lemma \ref{lemma Wexp bis}, there exists $\delta_0 \in (0,1)$ such that
for $\lambda \in [0, \lambda_0]$,
\begin{eqnarray}
\label{majoration 3}
\mathbb{P}_{\lambda} \left( |D_m^{\lambda}| > \frac{x}{2} \right)
&\leq& \mathbb{P}_{\lambda} \left( | \log W_n - \log W_m|^{1/2} > \frac{x^{1/2}n^{1/4}}{\sqrt{2}} \right) \notag \\
&\leq& C \delta_0^{m/2} x^{-1/2} n^{-1/4}.
\end{eqnarray}
From \eqref{majoration 3 bis} and \eqref{majoration 3} we have, for any  $\lambda \in [0, \lambda_0]$,
\begin{eqnarray}
\label{majoration 4}
\mathbb{P}_{\lambda} \left( |D_m^{\lambda}| > \frac{x}{2} \right)
&\leq& C \min \left( e^{ - \beta_2 x \sqrt{n} }, \;  \delta_0^{\sqrt{n}/2} x^{-1/2} n^{-1/4} \right).
\end{eqnarray}
Using (\ref{troncature lemme 4.2 ii}), (\ref{majoration 2}) and (\ref{majoration 4}), we get \eqref{ecart002} with
$\delta = \delta_0^{1/2}$. This
 ends the proof of the lemma.
\end{proof}

\subsection{Proof of Theorem \ref{cramer type theorem}}
\label{section proof}
We shall prove only the first assertion, the second one being proved in the same way.

For $0\leq x\leq 1$,
the theorem follows from the Berry-Esseen estimate in Theorem \ref{BE
theorem for BPRE}. So we assume that $1\leq x=o(\sqrt{n})$.
Using the change of measure \eqref{Def chmt measure},
for any $\lambda \in [ 0,  \lambda _{0}],$ we have%
\begin{eqnarray*}
\mathbb{P}\left( \frac{\log Z_{n}-n\mu }{\sigma \sqrt{n}}> x\right)
&=& L\left( \lambda \right) ^{n}\mathbb{E}_{\lambda } \left[ e^{-\lambda S_{n}}%
\mathds{1}( \log Z_{n}-n\mu > x\sigma \sqrt{n}) \right] .
\end{eqnarray*}%
Denote%
\begin{equation}
\label{Y_nlambda}
Y_{n}^\lambda =\frac{S_{n}-n\mu _{\lambda }}{\sigma _{\lambda}\sqrt{n}} \quad \text{and } \quad
V_{n}^\lambda =\frac{\log W_{n}}{\sigma_{\lambda }%
\sqrt{n}}.
\end{equation}%
Using the decomposition \eqref{decom-logZn}, centering and reducing $S_n$ under $\mathbb{P}_{\lambda}$, we get
\begin{eqnarray*}
&& \mathbb{P}\left( \frac{\log Z_{n}-n\mu }{\sigma \sqrt{n}} > x\right)  \\
&=&\exp \left( n \psi (\lambda) - n \mu_{\lambda} \right)\mathbb{E}_{\lambda } \left[ e^{-\lambda \sigma_{\lambda} \sqrt{n} Y_n^{\lambda} }\mathds{1}%
\left( Y_{n}^\lambda +V_{n}^\lambda > \frac{x\sigma \sqrt{n} - n ( \mu_{\lambda}  - \mu)}{\sigma_{\lambda} \sqrt{n} } \right) \right] ,
\end{eqnarray*}%
with $\psi$ defined in \eqref{def-phi}. It is well known that for $x=o(\sqrt{n})$ as $n\to \infty$, the equation 
\begin{equation}
\label{equation lambda x 0}
x\sigma \sqrt{n}= n (\mu_{\lambda} - \mu),
\end{equation}
has a unique solution $\lambda(x)$ which can be expressed as the power series
\begin{equation}
\label{lambda(x)}
\lambda(x) =\frac{t}{\sqrt{\gamma _{2}}}-\frac{\gamma _{3}}{2\gamma
_{2}^{2}}t^{2}-\frac{\gamma _{4}\gamma _{2}-3\gamma _{3}^{2}}{6\gamma
_{2}^{7/2}}t^{3}+\ldots
\end{equation}%
with $t=\frac{x}{\sqrt{n}}$ (see \cite{petrov} for details).
Choosing $\lambda =\lambda \left( x\right),$ it follows that
\begin{eqnarray}
\label{changement mesure x J}
\mathbb{P}\left( \frac{\log Z_{n}-n\mu }{\sigma \sqrt{n}}> x\right)
&=&\exp \left( n \psi (\lambda) - n \mu_{\lambda} \right)\mathbb{E}_{\lambda } \left[ e^{-\lambda \sigma_{\lambda }\sqrt{n}Y_{n}^\lambda  }\mathds{1} ( Y_{n}^\lambda +V_{n}^\lambda
 > 0) \right] \notag \\
&=&\exp \left( n \psi (\lambda) - n \mu_{\lambda} \right) I,
\end{eqnarray}%
where
\begin{eqnarray}
I &=&\mathbb{E}_{\lambda } \left[e^{-\lambda \sigma _{\lambda }\sqrt{n}%
Y_{n}^\lambda }\mathds{1} ( Y_{n}^\lambda +V_{n}^\lambda > 0) \right] \notag  \label{expression J}
\\
&=&\int  e^{-\lambda \sigma _{\lambda }\sqrt{n}%
Y_{n}^\lambda }\mathds{1} ( Y_{n}^\lambda +V_{n}^\lambda > 0) d \mathbb{P}_{\lambda}.
\end{eqnarray}%
Using the fact that
\[
 e^{-\lambda \sigma _{\lambda }\sqrt{n}%
Y_{n}^\lambda } = \lambda \sigma _{\lambda }\sqrt{n}%
\int_{\mathbb{R}} \mathds{1}(Y_{n}^\lambda < y) e^{-\lambda \sigma _{\lambda }\sqrt{n}%
y} dy
\]
and Fubini's theorem, we obtain
\begin{eqnarray}
I &=&\lambda \sigma _{\lambda }\sqrt{n}\int_{\mathbb{R}}e^{-\lambda
\sigma _{\lambda }\sqrt{n}y}\mathbb{P}_{\lambda }\left( Y_{n}^
\lambda<y, Y_{n}^\lambda +V_{n}^\lambda >0\right) dy  \notag
\label{J}.
\end{eqnarray}%
Obviously $I= I_+ + I_-$, with
\begin{equation*}
I_{+} =\lambda \sigma_{\lambda }\sqrt{n}\int_{0}^{\infty
}e^{-\lambda \sigma_{\lambda }\sqrt{n}y}\mathbb{P}%
_{\lambda }\left( Y_{n}^\lambda<y,Y_{n}^\lambda +V_{n}^\lambda >0\right) dy,
\end{equation*}
\begin{equation*}
I_{-} =\lambda \sigma_{\lambda }\sqrt{n}\int_{-\infty
}^{0}e^{-\lambda \sigma_{\lambda }\sqrt{n}y}\mathbb{P}_{\lambda }\left( Y_{n}^\lambda<y,Y_{n}^\lambda +V_{n}^\lambda >0\right) dy.
\end{equation*}%
We shall show that
\begin{equation}
\label{comparaison I I1}
I = I_1 (1 + O \left( \lambda )\right),
\end{equation}
where
\begin{equation} \label{def-I1}
I_1=\lambda \sigma _{\lambda }\sqrt{n} \int_{0}^{\infty} e^{-\lambda
\sigma_{\lambda }\sqrt{n}y} \Phi ([0, y ]) dy.
\end{equation}
By Lemma \ref{ecart iid} (i) we get by a straightforward computation that
\begin{eqnarray}
\label{borne J+}
\left| I_{+}-I_1 \right| \leq \frac{C}{\sqrt{n}}.
\end{eqnarray}%
By Lemma \ref{ecart iid} (ii), we have
\begin{eqnarray}
\label{J-}
 I_{-} &\leq & C \lambda \sigma_{\lambda }\sqrt{n}
         \int_{-\infty}^{0}e^{\lambda \sigma_{\lambda }\sqrt{n}|y|}
          \Big[ \left( |y| +\frac{1}{\sqrt{n}}\right) e^{-\beta_1 |y|  \sqrt{n}}  \nonumber \\
 & & \hspace{4,5cm} + \min \left( e^{-\beta_2 |y| \sqrt{n} }, \;  \delta^{\sqrt{n}} |y|^{-1/2} n^{-1/4} \right) \Big] dy. \nonumber
\end{eqnarray}%
Recall that, by (\ref{3 moments uniformes}), $\sigma_{\lambda}$ is bounded for $\lambda$ small enough
and, by (\ref{lambda(x)}), we have $\lambda \to 0$ as $n \to \infty$.  Then for $0 < \varepsilon <  \min(\beta_1, \beta_2)$,   we have
 $ \lambda \sigma_\lambda < \varepsilon $  for all $n$ large enough. Thus, by a straightforward calculation and by choosing $\varepsilon >0 $ small enough,
 it can be seen that
\begin{eqnarray*}
I_- &\leq& C \;  \lambda  \sqrt{n} \int_{-\infty}^{0}  \left( |y|+\frac{1}{%
\sqrt{n}}\right)  e^{-(\beta_1-\varepsilon )  \;  |y| \sqrt{n}} dy \\
&& + C \; \lambda  \sqrt{n} \int_{-\infty}^{-1}
    e^{-( \beta_2 - \varepsilon) \;  |y| \sqrt{n}} dy   \nonumber \\
 && + C \;  \lambda  \sqrt{n}  \int_{-1}^{0}    \delta^{\sqrt{n}} |y|^{-1/2} n^{-1/4} \;\,  e^{\varepsilon  \; |y | \sqrt{n}} dy     \\
&\leq&   \frac{ C \lambda}{\sqrt{n}} .
\end{eqnarray*}
By (\ref{lambda(x)}), we get, as $n \to \infty$,
\begin{equation}
\label{borne J-}
I_{-}=o\left( \frac{1}{\sqrt{n}}\right) .
\end{equation}%
From (\ref{borne J+}) and (\ref{borne J-}) it follows that
\begin{equation}
\label{borne J}
|I- I_1| \leq \frac{C}{\sqrt{n}}.
\end{equation}%
The intergal $I_1$  appears in the proof of the Cram\'er's large deviation expansion theorem for the i.i.d. case. 
For convenience, we state here some well known results concerning the asymptotic expansion of the cumulant generating function $\psi(\lambda)$ and of the integral $I_1.$ 
For details we refer the reader to \cite{petrov}.

\begin{lemma}
\label{resultats admis I}
Let $X$ be a  r.v. such that  $ \e [e^{\lambda_0 |X|}] < \infty$ for some $\lambda_0>0$. For
 $\lambda \in (-\lambda_0, \lambda_0)$, let  $\psi (\lambda) = \log \e [e^{\lambda X}]$,
$\mu_{\lambda} = \psi'(\lambda)$ and  $\sigma_\lambda= \psi''(\lambda)$. Set $\mu = \e X$.
Then  for $1\leq x = o(\sqrt n)$,  $\lambda = \lambda (x)$ solution of $(\ref{equation lambda x 0})$ and $n$ large enough,  we have:
 \begin{itemize}
 \item[(i)]  the cumulant generating function $\psi (\lambda) = \log \e [e^{\lambda X}]$ satisfies the  identity
\begin{eqnarray}
\label{terme principal deviation cramer 0}
 \frac{x^2}{2} + n ( \psi ( \lambda ) - \lambda \mu_{\lambda })&=& \frac{x^3}{\sqrt{n}} \mathscr{L} \left( \frac{x}{\sqrt{n}} \right) ,
\end{eqnarray}
where $\mathscr{L} (t)$  is the Cram\'er's series defined by \eqref{Cram series};
\item[(ii)]
 the integral $I_1$ defined by \eqref{def-I1}  satisfies the property that there exist   some positive  constants $C_1, C_2 >0$ such that
\begin{equation}
\label{encadrement I}
C_1 \leq  \lambda \sigma_{\lambda} \sqrt{n} I_1 \leq  C_2 \; ;
\end{equation}
moreover, the integral $I_1$ admits the following asymptotic expansion:
\begin{equation}
\label{I1 developpement}
I_1 =  \exp \left( \frac{x^2}{2} \right) \left[ 1 -
\Phi \left( x \right) \right] \left( 1 + O \left( \frac{x}{\sqrt{n}} \right) \right).
\end{equation}
 \end{itemize}
\end{lemma}
Now we can end the proof of Theorem \ref{cramer type theorem}.
By (\ref{borne J}), (\ref{encadrement I}) and (\ref{lambda(x)}),  we have
\begin{equation}
\label{I et I1 Cramer BPRE}
I=  I_1 \left(1 + O (\lambda) \right) = I_1 \left(1 + O \left( \frac{x}{\sqrt{n}} \right) \right). \nonumber
\end{equation}
Coming back to (\ref{changement mesure x J}) and using (\ref{I1 developpement}), we get
\begin{equation*}
\mathbb{P}\left( \frac{\log Z_{n}-n\mu }{\sigma \sqrt{n}}> x\right)
= \exp \left( \frac{x^2}{2} + n ( \psi ( \lambda) - \lambda \mu_\lambda)\right)(1- \Phi (x)) \left( 1+ O \left( \frac{x}{\sqrt{n}} \right) \right).
\end{equation*}%
Then, by (\ref{terme principal deviation cramer 0}), we obtain the desired  Cram\'er's large deviation expansion
\begin{equation}
\label{Cramer large deviation expansion BPRE}
\mathbb{P}\left( \frac{\log Z_{n}-n\mu }{\sigma \sqrt{n}}> x\right)
= \exp \left( \frac{x^3}{\sqrt{n}} \mathscr{L} \left( \frac{x}{\sqrt{n}} \right)  \right)(1- \Phi (x)) \left( 1+ O \left( \frac{x}{\sqrt{n}} \right) \right),
\end{equation}%
which ends the proof of the first assertion of Theorem \ref{cramer type theorem}.

\bibliographystyle{plain}
\bibliography{Bibliography_BPRE2}

\nocite{smith}

\nocite{athreya1971branching}

\nocite{athreya1971branching2}

\nocite{liu}

\nocite{hambly}

\nocite{huang_convergence_2014}

\nocite{bansaye2009large}

\nocite{boinghoff2010upper}

\nocite{bansaye2011upper}

\nocite{bansaye2013lower}

\nocite{kozlov2006large}

\nocite{bansaye2014small}

\nocite{petrov}

\nocite{cramer}

\nocite{afanasyev2012limit}

\end{document}